\documentclass[3p,12pt]{elsarticle}

\usepackage{amssymb}
\usepackage{graphicx}
\usepackage{color}
\usepackage{amsmath}
\usepackage{amssymb}

\usepackage{subcaption}
\usepackage{amsthm}

\DeclareGraphicsExtensions{.pdf,.png,.jpg}
\newtheorem{theorem}{Theorem}
\newtheorem{lemma}{Lemma}
\newtheorem{remark}{Remark}

\newcommand{\R}{\mathbb{R}}

\renewcommand{\max}{{\mathrm{max}}}

\newcommand{\norm}[1]{\left\| #1 \right\|}

\usepackage{dsfont}

\usepackage{array}
\newcolumntype{C}{>{$}c<{$}}
\newcolumntype{L}{>{$}l<{$}}
\newcolumntype{R}{>{$}r<{$}}
 {\left\lbrace\begin{array}{@{}l@{}}}%
 {\end{array}\right.}
 {\left\lbrace\begin{eqnarray}{@{}l@{}}}%
 {\end{eqnarray}\right.} 

\journal{Applied Numerical Mathematics}

\begin{document}

\begin{frontmatter}



\title{Numerical Methods for the Nonlocal Wave Equation of the Peridynamics}

\author[poliba]{G. M. Coclite}
\ead{giuseppemaria.coclite@poliba.it}

\author[dim]{A. Fanizzi}
\ead{alessandro.fanizzi@hotmail.it}

\author[dim]{L. Lopez} 
\ead{luciano.lopez@uniba.it}

\author[poliba]{F. Maddalena}
\ead{francesco.maddalena@poliba.it}

\author[dim]{S. F. Pellegrino}
\ead{sabrina.pellegrino@uniba.it}

\address[dim]{Dipartimento di Matematica, Universit\`a degli Studi di Bari Aldo Moro, via E. Orabona 4, 70125 Bari, Italy}
\address[poliba]{Dipartimento di Matematica, Politecnico di Bari, Via Re David, 70125, Bari, Italy.}


\begin{abstract} In this paper we will consider the peridynamic equation of motion which is described by a second order in time partial integro-differential equation. This equation has recently received great attention in several fields of Engineering because seems to provide an effective approach to modeling mechanical systems avoiding spatial discontinuous derivatives and body singularities.  
In particular, we will consider the linear model of peridynamics in a one-dimensional spatial domain. 
Here we will review some numerical  techniques to solve this equation and propose some new computational methods of higher order  in space;  moreover we will see  how to apply the methods studied for the linear model  to the nonlinear one. 
Also a spectral method for the spatial discretization of  the linear problem will be discussed.  Several numerical tests will be given in order to validate our results.

\end{abstract}


\end{frontmatter}

\section{Introduction}

Nonlocal continuum mechanics aims at modeling long-range interactions occurring in real materials, ruling several  phenomena like fracture instabilities, damage, defects, phase boundaries, etc. Capturing these effects is a long standing problem in continuum physics and different models have been proposed in literature (see \cite{Er,EE,K,Ku}). More recent studies show that nonlocal models based only on derivatives of integer order are not completely satisfactory to depict the nature of several phenomena and therefore, on the basis of physical and mathematical considerations, in order to model such situations, differential operators of fractional orders may be introduced~\cite{Garrappa2013,Garrappa2014,Garrappa2015,CDMV}. In \cite{Silling_2000} Silling introduced {\bf  peridynamics} as a nonlocal elasticity theory: he proposed a model describing the motion of a material body based on integro-differential partial equations, not involving spatial derivatives. 
The main idea underlying  peridynamic theory relies in assuming a force $f$, acting on a spatial region $V_x$, occupied by a material body, as the fundamental interaction between the particle  $x$ and the particle $\hat{x}$ belonging to $V_x$, which represents the peridynamic neighborhood of $x$.  This basic assumption also suggests that peridynamics could be suitable for multiscale material modeling (\cite{L,QYX,LSJ}).

We fix $[0,T]$ as the time interval under consideration. Let $V\subset \R^d$, with $d\in\{1,2,3\}$, be the rest configuration of a material body endowed with a mass density $\rho:V\times [0,T]\rightarrow \R_+$ and let $u:V\times [0,T]\rightarrow \R^d$ be the displacement field assigning at the particle having position $x\in V$
at time $t=0$ the new position $x+u(x, t)$ at time $t$.
Peridynamics postulates the existence of a long range internal force field, in place of the classical contact forces,
 hence,  the evolution of the material body is governed  by  
the following non-local version of the linear momentum balance: 
\begin{equation}\label{neq1}
\rho(x)u_ {tt}  (x,t)=\int_{V} f(\hat x-x,  u (\hat x,t)-u(x,t))d\hat x+ b(x,t),
\end{equation}
 usually enriched by the  initial conditions
\begin{equation}\label{eq1}
u(x,0)=u_0(x), \quad u_t(x,0)=v(x), \qquad x\in V,
\end{equation}
where $b(x,t)$ describes  the external  forces. The integrand  $f$ is called {\bf  pairwise  force  function}  and gives the  force  density per unit reference volume that the particle $\hat x$ exerts on the particle  $x$. It depends  on  the  material of  the  body  and, in particular, different forms of $f$ appear in literature depending on the characteristic of the material, see, for instance,~\cite{Emmrich_Puhst_2013,Silling_2000}.

In~\eqref{neq1}, the integral term sums up the forces that all particles in the volume $V$ exert on the particle $x$ and these  interactions  are  called {\bf bonds}. Setting
\begin{equation}\label{nonl1}
\xi = 
\hat x- x , \quad {\rm and \ }  \quad \eta= 
u(\hat x;t)-u(x;t), 
\end{equation}
we observe that $f$ has to satisfy the general principles of mechanics. Then, Newton's third law  and  the conservation  of  angular momentum deliver:
\begin{equation}
\label{nonl2}
f(-\xi,-\eta)= -f(\xi,\eta)\quad {\rm and \ }  \quad \eta\times f(\xi,\eta)=0.
\end{equation}

It is reasonable to assume  that there are no interactions between particles separated by a distance greater than a fixed value, namely, we require that there exists a positive constant $\delta$, called {\bf horizon}, such that

$$ |\xi|>\delta \Rightarrow f(\xi,\eta)=0, \qquad\text{for every $\eta$},$$
thus the 
integral  in (\ref{neq1}) can  be  understood  as 
$$\int_{V}  f (\hat x-x, \hat u (x,t)-u(x,t))d\hat x=\int_{V \cap B_\delta(x)}   f (\hat x-x, \hat u (x,t)-u(x,t))d\hat x , 
$$
 where $B_\delta(x)\subset \R^d$ denotes the open ball centered at $x$ with radius $\delta>0$ 
(see \cite {Emmrich_Puhst_2013}). 

In this paper we restrict our attention to the one-dimensional version of this theory,  for an homogeneous bar of infinite length, so that equation~\eqref{neq1} is replaced by 
\begin{equation}\label{perid21}
\rho(x) \  u_{tt}(x,t)  = \int_ {-\infty}^{\infty} f(\hat x-x, u(\hat x,t)-u(x,t)) d\hat x +b(x,t), \qquad x\in \R, \ t\ge 0, 
\end{equation}
and in particular we focus on the following {\bf linear peridynamic} model
\begin{equation}\label{perid1}
\rho\  u_{tt}(x,t)  = \int_ {-\infty}^{\infty} C(\hat x-x)(u(\hat x,t)-u(x,t)) d\hat x +b(x,t), \qquad x\in \R, \ t\ge 0, 
\end{equation}
where $\rho$ denotes the constant  mass density, $u$ the displacement field of the body, $b$ collects the external forces. The function $C$, called  {\bf  micromodulus function},  is a non negative even function, namely $C (\xi)=C(-\xi)$ with $\xi= \hat x-x$ .

The equation~\eqref{perid1} is associated to the  initial conditions
\begin{equation}\label{per1}
u(x,0)=u_0(x), \quad u_t(x,0)=v(x), \qquad x\in \R. 
\end{equation}

The aim of this paper is to review some numerical techniques  for the linear model and propose  new computational techniques based on accurate spatial discretizations together with trigonometric schemes for the time discretization.  For the linear model also a  spatial discretization by spectral techniques is studied.  Furthermore, we extend some of these methods to the nonlinear case.

The paper is organized as follows. In Section~\ref{sec:preliminaries}, we present the main theoretical results for this problem. In Section~\ref{sec:spatial} we discretize in space the equation~\eqref{perid1} by composite quadrature formulas. Spectral spatial discretization methods and their convergence are discussed in Section~\ref{sec:spectral}. Section~\ref{sec:time} is devoted to the time discretization techniques. In Section~\ref{sec:nonlin} we extend the numerical methods implemented for the linear model to the nonlinear model~\eqref{perid21}. Section~\ref{sec:numerical} is devoted to numerical tests, and finally, Section~\ref{sec:con} concludes the paper.

\section{Preliminary results}
\label{sec:preliminaries}

The study of well-posedness of the peridynamic problem crucially depends on the constitutive assumptions made on the pairwise force $f$ and several results appear in literature~\cite{Emmrich_Puhst_2013, CDMV, Emmrich_Puhst_2015}. In what follows, we briefly recall the main results. Identifying $u:V\times [0,T]\to \R^d$ with $\bar u:[0,T]\to X$, for a function space $X$
which is a subset of  the maps from $\bar V$ into $\R^d$  defined by $[\bar u(t)](x)=u(x,t)$, and denoting again $\bar u$ with $u$, we derive the equivalent abstract formulation of the problem (\ref {neq1}):
\begin{equation}\label{abstrace}
u''(t)=g(u(t),t), \ t\in[0,T], \qquad u(0)=u_0, \ u'(0)=v,
\end{equation}
where $g$ is defined as $g(v,t)=(Kv+b(t))/\rho$ and the integral operator $K$ is given by
\begin{equation}\label{eq:K}
(Ku)(x):= \int_{V \cap B_\delta(x)}f(\hat x-x,u(\hat x)-u(x))\ d\hat x.
\end{equation}

Let  $C(V)^d$ be the space of continuous $\R^d$ valued functions defined on $V\subset \R^d$. Let us  recall the following result concerning with the nonlinear model. 

\begin{theorem}\label{Tv1}{\rm (see~\cite{Emmrich_Puhst_2013})}. Let $u_0,v\in C(\overline{V})^d$ and $ b\in C([0,T];
C(\overline{V})^d)$. Assume that $f: \overline{B_\delta(0)}\times
\R^d\rightarrow \R^d$  is a continuous function and that there exists a
nonnegative function $\ell\in L^1(B_\delta(0))$
such that for all $\xi\in \R^d$ with $\vert \xi\vert\leq \delta$ and $\eta,\hat \eta\in \R^d$ there holds
$$\vert f(\xi,\hat \eta)-f(\xi,\eta)\vert\leq\,\ell(\xi)\vert\hat \eta-\eta\vert.$$ 
Then, the integral  operator $K:C(\overline{V})^d\rightarrow\,\R$ 
is well-defined and Lipschitz-continuous, and the initial-value problem \eqref{abstrace} is globally well-posed
with  solution $u\in C^2([0,T];C(\overline{V})^d)$.
\end{theorem} 

\medskip
For a {\bf microelastic} material (see \cite{Silling_2000}), the pairwise force function $f(x,\hat x,\eta)$ may 
be derived from a scalar-valued function $w(x,\hat x,\eta)$ called {\bf  pairwise potential function}    
 (see \cite{Emmrich_Weckner_2007}),  such that
\begin{equation}\label{cross52}
f(x,\hat x,\eta)=\nabla_ {\eta}  w(x,\hat x,\eta),
\end{equation}
and the peridynamic equation  (\ref{neq1})  derives from the variational problem: find 
\begin{equation}\label{cross72}
u={\rm arg\ min} \ J(u)\ , \qquad  J(u)=\int^T_0\int_V e (x, u(x,t),t) dxdt,
\end{equation}
where $e = e_{kin} - e_{el} - e_{ext}$ is the Lagrangian density, and incorporates the kinetic energy density, the elastic energy density and the density due to the external force density, given respectively by 
$$e_{kin}=   \frac {1}{2} \rho(x)\   u_t^2(x,t),\quad 
 e_{el}=  \frac {1}{2} \int_ {V}w(x,\hat x, u(\hat x,t)-u(x,t))d\hat x\ , \quad  e_{ext}=-b(x,t)u(x,t).$$

%
%

In particular, in the one-dimensional linear peridynamic model~\eqref{perid1}, the potential function is given by
\[w(x,\hat x,\eta) = \frac{1}{2} C(\hat x - x)\eta^2,\]
and we have the following result.


\begin{theorem}\label{T1}{\rm (see \cite{Emmrich_Weckner_2007})}.   Assume the function $C\in C^2(\R)$. Then for any initial value $u_0$ and $v$ in $C^0(\R)$ and for any $T>0$, the Cauchy problem (\ref{perid1})-(\ref{per1})  admits a unique solution $u\in C^2([0,T];C(\R))$.  Moreover for such a problem the total energy remains constant if the external forces are autonomous, i.e. $b$ does not depend on $t$:
$$    \frac {d}{dt} \left(E_ {kin}(t)+E_{el}(t)+E_{ext}(t)\right)=0, \qquad t\ge 0,$$
where $E_{i}(t) = \int_V e_{i}(x,u,t)\,dx$, for $i\in\{kin,el,ext\}$.
Otherwise,  for all $\nu>0$ and $t>0$, the following inequality holds true
\begin{gather*}  e_ {kin}(t)+e_{el}(t)+\nu\int_0^ t e^ {\nu(t-s)}e_{ext}(s)ds\\
\le  e^ {\nu t}   ( e_ {kin}(0)+e_{el}(0))+   \frac {1}{2\nu} \int_0^ t  \int_ {-\infty}^{\infty}    \frac { e^ {\nu(t-s)}} {\rho}|b(x,t)|^2 dx ds.
\end{gather*}
\end{theorem}

Additionally, in~\cite{CDMV}, the authors proved the well-posedness of the nonlinear peridynamic equation assuming very general constitutive assumptions in the framework of fractional Sobolev spaces.

 Moreover, we have to observe that the connections  between the linear 1D peridynamic equation (\ref{perid1}) and the linear  1D classical wave equation are well known  (see for example \cite{Emmrich_Weckner_2005},~\cite{BAC}).  Indeed,  if we consider  $u_0(x)=U\exp[(-x/L)^2]$, $v(x)=0$ with $U$ and $L$ suitable constants,  and  the micromodulus function 
\begin{equation}\label{microm2}
C(\hat x-x)=4E\exp[-(\hat x -x)^2/l^2]/(l^3\sqrt \pi),\qquad \hat x, x\in \R\ ,
\end{equation}
where  $E$ denotes the Young modulus, and $l>0$ a length-scale parameter, then for $l\to 0$,   (\ref{perid1})  becomes the wave equation of the classical elasticity theory, that is:
\begin{equation}\label{wave1}
\rho\  u_{tt}(x,t)  =E u_{xx}(x,t) +b(x,t), \qquad x\in \R, \ t\ge 0\ , 
\end{equation}
Therefore, $l$ can be seen as a degree of nonlocality.

\section{Spatial discretization by composite quadrature formulas}
\label{sec:spatial}
A common way to approximate the solution of  the equation  (\ref{perid1})  is to apply a quadrature formula to discretize in space,  
in order to obtain a second order finite system of ordinary differential equations which  has to be integrated  in time. The order of accuracy of this formula will provide the discretization error in the space variable. Here we describe briefly this approach.

Let $N>0$ be an  even (large)  integer, $h>0$ be the spatial step size.  Let us discretize the spatial domain $(-\infty,\infty)$ by a compact set $[-D,D]$, for some positive large constant $D$, and such interval by means of the points $x_j=-D+jh=-D+j\frac{2D}{N}$, for $j=0,\ldots,N,$ and use a quadrature formula 
 of order $s$ (that is  the error of which is  $O(h^s)$) on these points,  then:   

\begin{equation}\label{ped9}
 \int_ {-\infty}^{\infty} C(\hat x-x)(u(\hat x,t)-u(x,t)) d\hat x \approx    h   \sum_{j=0}^N w_j  C(x_j-x)(u(x_j,t)-u(x,t)), 
\end{equation}
where $w_j$ are the weights of the formula. Then,  the equation (\ref{perid1})  may be approximated at each  $x=x_i$ for $i=0,\ldots, N$ by
\begin{equation}\label{perid9}
\rho u_{tt}(x_i,t)  \approx h \sum_{j=0}^N w_jC(x_j-x_i)(u(x_j,t)-u(x_i,t)) +b(x_i,t), \qquad \ t\ge 0.
\end{equation}
Let  $K=(k_{ij})$ be  the $(N+1)\times (N+1)$ {\bf stiffness matrix} whose generic entry is given by
$$k_{ij}=\alpha_i \delta_{ij}-w_jC_{ij}, $$
for $i,j=0,\ldots,N,$ with $C_{ij}=C(x_j-x_i)$,  $\alpha_i=\sum_{k=0}^N w_k C_{ik}$, and $\delta_{ij}$ is the Kronecker Delta.

In this case, the $(i+1)-th$ row of $K$ is given by
$$ [-w_0C_{i0}\  \ldots \ -w_ {i-1}   C_{ii-1}  \quad  (\alpha_i-w_iC_{ii})\qquad   -  w_{i+1}C_{ii+1}\  \ldots\ -w_NC_{iN} ],$$ 
for $i=0,\ldots, N$, and even if $C_{ij}=C_{ji}$,   {\bf  the matrix  $K$  is not symmetric, unless the weights are constant with respect to}  $j$,  i.e. $w_j=w$ for all  $j=0,\ldots,N$. Then, the  $(i+1)-th$ row of $K$  becomes 
$$ w [-C_{i0}\  \ldots \  C_{ii-1}  \quad    \sum_{k=0,k\neq  i}^N  C_{ik}  \quad   -  C_{ii+1}\  \ldots\  C_{iN}].$$ 

This is the case of  {\bf   the composite  midpoint rule}: here, we approximate the spatial domain $(-\infty,\infty)$ by the interval $[-(N+1)h/2,(N+1)h/2]$ and the points of the discretization $x_j^{MR}$ are taken as the midpoints of the subintervals $[-(N+1)h/2+jh,-(N-1)h/2+jh]$, for $j=0,\ldots,N$. 
For a sufficiently smooth problem (i.e. $C$ and $u$   bounded smooth functions), this formula is of the second order 
of accuracy in space, that is the error is $O(h^2)$, with constant  weigths given by  $w_j=1$ for $j=0,\ldots,N$  (see for instance \cite {Emmrich_Weckner_2007, Silling_Askari_2005}).


Instead, under more regularity on $C$ and $u$ , if we employ   {\bf   the composite Gauss two points formula}~\cite{LAURIE2001}, which has fourth order accuracy,  we can derive a symmetric stiffness matrix $K$.
Let us briefly recall this formula.  We fix $M>0$ and to evaluate the integral of a sufficiently  smooth function $\psi(x)$ we approximate $(-\infty,\infty)$ by the interval $[-D,D]$ and consider a partition of such interval given by the sequence $\tilde x_j=-D+jh$  for $j=0,\ldots,M$,  where $h=2D/M=(\tilde x_M-\tilde x_0)/M$. Then on each subinterval $[\tilde x_{j-1}, \tilde x_{j}]$
 for $j=1,\ldots,M,$ the formula uses  two points where the  function $\psi(x)$ is evaluated, that is:

\begin{equation}\label{G2p}
\int_ {\tilde x_0}^{\tilde x_M} \psi(x)dx \approx\frac{h}{2} \sum_ {j=1}^{M} \left[ \psi (m_j^- )+\psi  (m_j ^+)\right],
\end{equation}
where
$$  m_j = \frac {\tilde x_{j-1}+\tilde x_j}{2}  ,\qquad   m^-_j= m_j - \frac{h}{2\sqrt 3},\qquad m^+_j= m_j +\frac{h}{2\sqrt 3} , $$
for $j=1,\ldots,M$. Setting
\[ x_j  =
  \begin{cases}
    m^-_{   \frac{j+1}{2} },    & \quad \text{if } j \text{ is even},\\\\
    \ m^+_{   \frac{j+1}{2}  },  & \quad \text{if } j \text{ is odd},
  \end{cases}
\]
 for $j=0,\ldots,N$ with $N=2M-1$, then  we can rewrite the quadrature formula (\ref{G2p}) in the following way:
$$\int_ {x_0}^{x_M} \psi(x)dx\approx\frac{h}{2} \sum_ {j=1}^{M} \left[ \psi (m_j^-  )+\psi  (m_j ^+)\right ]= \frac{h}{2} \sum_ {j=0}^{N}  \psi ( x_j  ) ,$$
in order to have a formula on $N+1$ points and  constant  weights given by   $w_j=\frac{1}{2} $ for $j=0,1,\ldots,N$. 

\begin{remark}
Using the composite midpoint rule, or the composite Gauss two points formula, the stiffness matrix $K=(k_{ij})$   (where $k_{ij}=\alpha_i \delta_{ij}-w_jC_{ij} $)   is of size  $(N+1)\times (N+1)$  and such that
$$k_{ii}= - \sum_{j=0,j\neq i }^N k_{ij}, \ {\rm for \ all\ }i=0,\ldots N,$$
with $k_{ii}>0$; hence $K$ is a positive semidefinite matrix  with nonnegative  eigenvalues.

In  general $K$ is not sparse because of the infinite horizon, however, its entries may decrease when their  distance from the diagonal increases. For instance, if the micromodulus function is the one  in 
(\ref{microm2})   then a banded approximation 
of  $K$ which preserves the  accuracy of the numerical procedure can be used instead of $K$.

In case of finite horizon $\delta>0$ (see \cite{Bobaru_2009, Silling_Askari_2005}),  that is $C(x-\hat x)=0$,  when $ |x-\hat x|>\delta$, 
then $K$ has a banded  structure with the size of the band depending on $\delta$ and $h$. In this case we set  $r= \left\lfloor \delta/h\right\rfloor$ in order to have that $K$ is a $r$-band matrix. 

Thus the stiffness matrix $K$ results to be symmetric with the $(i+1)-th$ row  given by
$$w [\ 0 \ldots 0\ -  C_{ii-r}\ldots\ - C_{ii-1}  \quad    \sum_{k=-r,k\neq  i}^r  C_{ik}  \quad   -  C_{ii+1}\  \ldots -  C_{ii+r}\   0\ldots 0\  ]$$ 
for $i=0,\ldots, N$. 
\end{remark}


\subsection{The semidiscretized problem}
We set 
$$U(t)=  [U_0(t),U_1(t),\ldots, U_N(t)],$$
where the component $U_j(t)$ denotes an approximation of the solution at  the spatial node $x_j$, i.e. $U_j(t)\approx u(x_j,t)$ for $j=0,\ldots,N$, and
$$B(t)= \frac{1}{\rho}    [b(x_0,t),\ldots,  b(x_N,t)]^T.$$
Then, the equation (\ref{perid1}) may be approximated by the following second order differential system:
\begin{equation}\label{perid2}
U''(t)+ \Omega^2U(t)=B(t),
\end{equation}
with $\Omega^2= \displaystyle   \frac{h}{\rho} K$ (or  $\Omega^2= \displaystyle   \frac{hw}{\rho} K'$, where $K'$ depends only on the micromodulus function $C$), where $K$ is a positive semidefinite matrix, and with the initial conditions
$$U_0=[u_0(x_0),\ldots,u_0(x_N)]^T\quad {\rm and}\quad V_0=[v(x_0),\ldots,v(x_N)]^T.$$

\begin{remark}\label{rem1}
In order to avoid computational problems, particularly, when we will  consider trigonometric schemes where the square root $\Omega$ of $\Omega^2$ is required or the inverse of $\Omega$ is necessary,  we regularize the matrix $\Omega^2$ by adding a diagonal matrix of the form $h^s I$,  where $s$ is the order of accuracy of the quadrature formula used (see also \cite{Benzi_Liu_2007}, pag. 1979). 
Notice that choosing a perturbation having the same order of the accuracy of the quadrature formula, we do not affect the accuracy of the numerical solution.
With this choice, the matrix $\Omega^2$ will be symmetric and positive definite, and when it will be necessary  we can compute its square root $\Omega$  which will be unique, symmetric and positive definite; in particular the eigenvalues of $\Omega^2$ close to zero will be increased in $\Omega$.
\end{remark}
\begin{remark} The total energy $\mathcal{E}(t)$ of the semidiscretized system~\eqref{perid2} is the sum of the kinetic $\mathcal{E}_{kin}(t)$,  elastic $\mathcal{E}_{el}(t)$  and external $\mathcal{E}_{ext}(t)$ energy:
\begin{equation}\label{ene1}
\mathcal{E}(t)=\mathcal{E}_{kin}(t)+ \mathcal{E}_{el}(t)+ \mathcal{E}_{ext}(t), \qquad {\rm for\ } \ t\ge 0,
\end{equation}
with
\begin{equation}\label{ene2}
\mathcal{E}_{kin}(t)=\frac{1}{2} [U'(t)]^TU'(t), \quad \mathcal{E}_{el}(t)=\frac{1}{2} [U(t)]^T \Omega^2  U(t),\quad \mathcal{E}_{ext}(t)=- [U(t)]^TB(t).
\end{equation}
It is trivial to prove that  if the problem is autonomous (that is $b(x,t)=b(x)$) then $\mathcal{E}(t)=\mathcal{E}(0)$, for all $t\ge 0,$ while for nonautonomous problems, 
the semidiscretized energy has a behavior similar to the one in Theorem~\ref{T1}.  

However, even if the total energy $E(t)$ and the semidiscretized energy $\mathcal{E}(t)$ are constant in time, we have that 
$$|E(t)-\mathcal{E}(t)|=|E_0-\mathcal{E}_0|=O(h^s),$$
where $s$ is the accuracy of the quadrature formula used.

\end{remark}

The system (\ref{perid2}) is equivalent to the following first order differential system
\begin{equation}\label{perid6}
  \begin{pmatrix} U'  \\ V' \end{pmatrix} =     \begin{pmatrix} 0 &I\\ -\Omega^2& 0  \end{pmatrix}      \begin{pmatrix}  U \\ V\end{pmatrix} +
   \begin{pmatrix} 0 \\B(t)\end{pmatrix},
\end{equation}
where $V=U'$,  with the initial conditions $U_0$ and $V_0$. The exact solution of~\eqref{perid6} may be written as (see~\cite{Hochbruck2010})
\begin{equation}\label{perid16}
\begin{pmatrix} U(t) \\ V(t) \end{pmatrix} =  \exp(tA)   \begin{pmatrix} U_0\\ V_0 \end{pmatrix}  +\int^t_0  \exp[(t-s)A]   \begin{pmatrix} 0 \\ B(s) \end{pmatrix}ds,
\end{equation}
with $A= \begin{pmatrix} 0 &I\\ -\Omega^2& 0  \end{pmatrix} $.




\section{Spectral semi-discretization in space}
\label{sec:spectral}

Spectral spatial discretization is often obtained by means of a Fourier series expansion (with respect to the space variable) of the solution $u(x,t)$ of the partial differential equation studied (see for instance \cite{Emmrich_Weckner_2007_2}), followed by a numerical approximation obtained a truncation of the series expansion.
We now consider spectral semi-discretization in space with equidistant collocation points.

Let $N > 0$ be an even large integer and $h>0$ be the space step. We approximate the spatial domain $\R$ by a compact set $D=[-M\pi,M\pi]$, with $M>0$ and the boundary conditions by the periodic boundary conditions on $[-M\pi, M\pi]$, that is $u(-M\pi,t) = u(M\pi,t)$. It is expected that the initial-boundary valued problem can provide a good approximation to the original initial-valued problem as long as the solution does not reach the boundaries.
We assume that $C(x,\hat x) = 0$ for $x,\hat x \notin[-M\pi,M\pi]$. We discretize the compact set by means of the equidistant points $x_j = j h = j\frac{M\pi}{N}$, for $j= -N, \dots, N - 1$.

We seek an approximation in form of real-valued trigonometric polynomials
\begin{equation}
\label{eq:Fourier}
u_N(x,t) = \sum_{|k|\le N}{\tilde{u}_k(t)\,e^{\Im kx}},\qquad v_N(x,t) = \sum_{|k|\le N}{\tilde{v}_k(t)\,e^{\Im kx}}
\end{equation}
where $\tilde{v}_k(t) = \frac{d}{dt}\tilde{u}_k(t)$ and $\Im $ is the imaginary unit  $\Im = \sqrt {-1}$,.

Notice that $\tilde{u}_k(t)$, for all $k$, are unknown coefficients and for such method they represent the discrete Fourier transform
\begin{equation}
\label{eq:DFT}
\tilde{u}_k(t) = \frac{1}{2N\ c_k}\sum_{j=-N}^{N-1} u(x_j,t)e^{-\Im  kx_j},\quad k=-N,\dots,N,
\end{equation}
where
\begin{equation*}
c_k=
\begin{cases}
2,\quad&\text{if }k=\pm N,\\
1,\quad&\text{otherwise}.
\end{cases}
\end{equation*}

 Substituting~\eqref{eq:Fourier} in~\eqref{perid1} and in~\eqref{eq1}, we obtain
\begin{align*}
\sum_{|k|\le N}\rho \tilde{u}_k''(t)e^{\Im kx}&=\int_{-\infty}^{\infty } C( \hat x-x) \left( \sum_{|k|\le N}{\tilde{u}_k(t)\,e^{\Im  k \hat x}}- \sum_{|k|\le N}{\tilde{u}_k(t)\,e^{\Im kx}}    \right) d\hat x  + \sum_{|k| \le N}\tilde{b}_k(t)e^{\Im kx}=\\
&=\sum_{|k| \le N}  \left ( \int_{-\infty}^{\infty } C(\hat x-x)\left(e^{\Im k \hat x}-e^{\Im k x}\right) d\hat x  \right)   \tilde{u}_k(t) + \sum_{|k| \le N}\tilde{b}_k(t)e^{\Im kx}=\\
&=\sum_{|k| \le N}\left(\left(\int_{-\infty}^{\infty} C(\hat x -x )\left(e^{\Im k(\hat x-x)}-1\right)d\hat x\right)\tilde{u}_k(t)+ \tilde{b}_k(t)\right)e^{\Im kx},
\end{align*}
and
\begin{equation*}
u_0(x)=\sum_{|k|\le N}\tilde{u}_{0,k} e^{\Im k x},\qquad v(x) = \sum_{|k|\le N}\tilde{v}_{0,k} e^{\Im k x}.
\end{equation*}

Therefore, the $2N+1$ independent frequencies $\tilde{u}_k(t)$ are the solutions of the following set of Cauchy problems:

\begin{equation}
\label{eq:ode}
\begin{cases}
\tilde{u}_k''(t) +   \frac{1}{\rho} \omega^2_k \tilde{u}_k(t) = \frac{1}{\rho}\tilde{b}_k(t),\\
\tilde{u}_k(0)=\tilde{u}_{0,k},\quad \tilde{u}'_k(0) = \tilde{v}_{0,k}\ ,
\end{cases}\quad k = -N, \dots, N,
\end{equation}
where
\begin{equation}
\label{eq:w}
\omega^2_k=\int_{-\infty}^{\infty} C(\hat x-x)\left(1-e^{\Im k(\hat x-x)} \right) d \hat x.
\end{equation}

 We notice that $\omega^2_k$  is real, in fact, setting $\xi=\hat x- x$ and observing that $ { \cal  C}     (\xi)= { \cal  C}     (-\xi)$ we can  easily  prove that



\[ \omega^2_k=  2 \int_{0}^\infty     { \cal  C}     (\xi)  \left( 1-\cos k\xi \right)    d  \xi. \]

The ODE system~\eqref{eq:ode} can be solved by a numerical method. Finally, we can obtain the solution in the physical space by using~\eqref{eq:Fourier}.

\subsection{Convergence of the Semi-Discrete Scheme}

This section is devoted to the study of the convergence of the spectral semi-discrete scheme. Throughout this section, $L$ denotes a generic constant. We use $(\cdot,\cdot)$ and $\norm{\cdot}$ to denote the inner product and the norm of $L^2(D)$, respectively, namely
\[
(u,v) = \int_{D} u(x)v(x)\,dx,\quad \norm{u}^2 = (u,u).
\]
Let $S_N$ be the space of trigonometric polynomials of degree $N$,
\[
S_N = \text{span}\left\{e^{\Im kx}|-N \le k\le N\right\},
\]
and $P_N: L^2(D) \to S_N$ be an orthogonal projection operator
\[
P_Nu(x) = \sum_{|k|\le N} \tilde{u}_k e^{\Im k x},
\]
such that for any $u\in L^2(D)$, the following equality holds
\begin{equation}
\label{eq:orthogonal}
(u-P_Nu,\varphi) = 0,\quad\text{for every $\varphi\in S_N$}.
\end{equation}
The projection operator $P_N$ commutes with derivatives in the distributional sense:
\[
\partial_x^q P_Nu = P_N\partial_x^q u,\quad\text{and}\quad\partial_t^q P_Nu = P_N\partial_t^q u.
\]
We denote by $H^s_p(D)$ the periodic Sobolev space and by $X_s = \mathcal{C}^1\left(0,T; H^s_p(D)\right)$ the space of all continuous functions in $H_p^s(D)$ whose distributional derivative is also in $H_p^s(D)$, with norm
\[
\norm{u}_{X_s}^2 = \max_{t\in[0,T]}\left(\norm{u(\cdot,t)}^2 + \norm{u_t(\cdot,t)}^2\right),
\]
for any $T> 0$.

The semi-discrete Fourier spectral scheme for~\eqref{perid1}-\eqref{per1} with periodic boundary conditions is
\begin{align}
\label{eq:scheme}
\rho u_{tt}^N &= P_N g(u^N) + P_N b(x,t),\\
\label{eq:initial_scheme}
u^N(x,0) &= P_N u_0(x),\quad u_t^N(x, 0) = P_N v(x),
\end{align}
where $u^{N}(x,t)\in S_N$ for every $0\le t\le T$, and $g(u)$ denotes the integral operator of~\eqref{perid1}, namely
\begin{equation}
\label{eq:g}
g(u(x,t)) = \int_{D} C(\hat x - x)\left(u(\hat x, t) - u(x, t)\right)\, d\hat x, \quad x\in D,\,0\le t\le T.
\end{equation}

To obtain the convergence of the semi-discrete scheme, we need of the following lemma.
\begin{lemma}[see~\cite{Canuto}]
\label{lm:sobolev}
For any real $0\le \mu\le s$, there exists a constant $L$ such that
\begin{equation}
\label{eq:sobolev}
\norm{u-P_N u}_{H_p^{\mu}(D)} \le L N^{\mu-s}\norm{u}_{H^s_p(D)}, \quad\text{for every $u\in H_p^{s}(D)$}.
\end{equation}
\end{lemma}

Now we can prove the following theorem.
\begin{theorem}
\label{th:convergence}
Let $s\ge 1$, $u(x,t)\in X_s$ be the solution of the initial-valued problem~\eqref{perid1}-\eqref{per1} with periodic boundary conditions and $u^N(x,t)$ be the solution of the semi-discrete scheme~\eqref{eq:scheme}-\eqref{eq:initial_scheme}. If $C\in L^{\infty}(D)$, then, there exists a constant $L$, independent on $N$, such that
\begin{equation}
\label{eq:order_conv}
\norm{u-u^N}_{X_1} \le L(T) N^{1-s} \norm{u}_{X_s},
\end{equation}
for any initial data $u_0$, $v\in H_p^s(D)$ and for any $T > 0$.
\end{theorem}

\begin{proof}
Let $s\ge 1$. Using the triangular inequality, we have
\begin{equation}
\label{eq:triangular}
\norm{u-u^N}_{X_1} \le \norm{u-P_N u}_{X_1} + \norm{P_N u - u^N}_{X_1}.
\end{equation}
Lemma~\ref{lm:sobolev} implies
\[
\norm{(u-P_N u)(\cdot,t)}_{H^1_p(D)} \le L N^{1-s}\norm{u(\cdot,t)}_{H_p^s(D)},
\]
and
\[
\norm{(u-P_N u)_t(\cdot,t)}_{H_p^1(D)} \le L N^{1-s}\norm{u_t(\cdot,t)}_{H_p^s(D)}.
\]
Therefore,
\begin{equation}
\label{eq:1term}
\norm{(u-P_N u)_t}_{X_1} \le L N^{1-s}\norm{u_t}_{X_s}.
\end{equation}
Subtracting~\eqref{eq:scheme} from~\eqref{perid1} and taking the inner product with $\left(P_N u - u^N\right)_t \in S_N$, we have
\begin{equation}
\label{eq:difference}
\begin{split}
0=&\underbrace{\int_{D} \rho\left(u_{tt}(x,t)-u_{tt}^N(x,t)\right)\left(P_N u(x,t) - u^N(x,t)\right)_t\,dx}_{{}=:I_1}\\
 &- \underbrace{\int_{D}\left(g(u(x,t))-P_N g(u^N(x,t))\right)\left(P_N u(x,t)- u^N(x,t)\right)_t\,dx}_{{}=:I_2}\\ 
&- \underbrace{\int_{D}\left(b(x,t)-P_N b(x,t)\right)\left(P_N u(x,t) - u^N(x,t)\right)_t\,dx}_{{}=:I_3}.
\end{split}
\end{equation}
The orthogonal condition~\eqref{eq:orthogonal} implies that
\[
\int_{D}\left(u_{tt}(x,t) - P_N u_{tt}(x,t)\right)\left(P_N u(x,t)-u^N(x,t)\right)_t\,dx = 0,
\]
and
\[
\int_{D}\left(b(x,t)-P_N b(x,t)\right)\left(P_N u(x,t) - u^N(x,t)\right)_t\,dx =0.
\]
Thus,
\begin{equation}
\label{eq:I1}
\begin{split}
I_1 &= \int_{D}\rho\left(u_{tt}(x,t) - P_N u_{tt}(x,t)\right)\left(P_N u(x,t)-u^N(x,t)\right)_t\,dx\\
& + \int_{D} \rho\left(P_N u_{tt}(x,t)-u^N_{tt}(x,t)\right)\left(P_N u(x,t)-u^N(x,t)\right)_t\,dx\\ 
&= \frac{\rho}{2}\frac{d}{dt}\norm{(P_N u - u^N)_t(\cdot,t)}^2_{H_p^1(D)},
\end{split}
\end{equation}
and $I_3 =0$.

Now we focus on $I_2$. Thanks to~\eqref{eq:orthogonal}, we have
\[
\int_{D}\left(g(u^N(x,t))-P_N g(u^N(x,t))\right)\left(P_N u(x,t)-u^N(x,t)\right)_t\,dx =0.
\]
Since $u(\cdot,t)$, $u^N(\cdot,t)\in H_p^1(D)$, there exists $L > 0$ such that
\[
\norm{(u-u^N)(\cdot,t)}^2_{H_p^1(D)} \le 2\left(\norm{u(\cdot,t)}^2_{H_p^1(D)} + \norm{u^N(\cdot,t)}^2_{H_p^1(D)}\right) \le L.
\]
As a consequence, since $C\in L^{\infty}(D)$ and using the Cauchy's inequality, we obtain
\begin{equation}
\label{eq:I2}
\begin{split}
I_2 &=\int_{D} \left(g(u(x,t))-g(u^N(x,t))\right)\left(P_N u(x,t)-u^N(x,t)\right)_t\,dx\\
&=\int_{D}\int_{D}C(\hat x - x)\left(u(\hat x,t)-u(x,t)-u^N(\hat x,t) + u^N(x,t)\right)\left(P_N u(x,t)-u^N(x,t)\right)_t\,d\hat x dx\\
&\le L \int_{D}\left(u(x,t)-u^N(x,t)\right)\left(P_N u(x,t)-u^N(x,t)\right)_t\,dx\\
& + \frac{1}{2}\norm{(u-u^N)(\cdot,t)}^2_{H^1_p(D)}\int_{D}\left(u(x,t)-u^N(x,t)\right)\left(P_N u(x,t)-u^N(x,t)\right)_t\,dx\\
&\le L \norm{(u-u^N)(\cdot,t)}^2_{H_p^1(D)} + L \norm{(P_N u - u^N)_t(\cdot,t)}^2_{H_p^1(D)}.
\end{split}
\end{equation}
Substituting~\eqref{eq:I1} and~\eqref{eq:I2} in~\eqref{eq:difference}, we have
\begin{equation}
\label{eq:substitution}
\frac{\rho}{2}\frac{d}{dt}\norm{(P_N u - u^N)_t(\cdot,t)}^2_{H_p^1(D)} \le L \norm{(u-u^N)(\cdot,t)}^2_{H_p^1(D)} + L \norm{(P_N u - u^N)_t(\cdot,t)}^2_{H_p^1(D)}.
\end{equation}
Adding to both sides of equation~\eqref{eq:substitution} the term
\[
\frac{1}{2}\frac{d}{dt}\norm{(P_N u - u^N)(\cdot,t)}^2_{H_p^1(D)} =\int_{D}\left(P_N u (x,t) - u^N(x,t)\right)\left(P_N u (x,t) - u^N(x,t)\right)_t\,dx,
\]
we obtain
\begin{equation*}
\begin{split}
&\frac{d}{dt}\left(\norm{(P_N u - u^N)_t(\cdot,t)}^2_{H_p^1(D)} + \norm{(P_N u - u^N)(\cdot,t)}^2_{H_p^1(D)}\right)\\
&\le L\left(\norm{(P_N u - u^N)_t(\cdot,t)}^2_{H_p^1(D)} + \norm{(P_N u - u^N)(\cdot,t)}^2_{H_p^1(D)} + \norm{(u-P_N u)(\cdot,t)}^2_{H_p^1(D)}\right).
\end{split}
\end{equation*}
Since $\norm{(P_N u - u^N)_t(\cdot,0)}_{H_p^1(D)} = 0$ and $\norm{(P_N u - u^N)(\cdot,0)}_{H_p^1(D)} = 0$, Lemma~\ref{lm:sobolev} and Gronwall's inequality imply that
\begin{align*}
\biggl(\norm{(P_N u - u^N)_t(\cdot,t)}^2_{H_p^1(D)} &+ \norm{(P_N u - u^N)(\cdot,t)}^2_{H_p^1(D)}\biggr)\\
&\le\int_0^t e^{L(t-\tau)}\norm{(u-P_N u)(\cdot,\tau)}^2_{H_p^1(D)}\,d\tau\\
&\le L(T)N^{2-2s}\int_0^t\norm{u(\cdot,\tau)}^2_{H_p^1(D)}\,d\tau.
\end{align*}
Thus,
\begin{equation}
\label{eq:2term}
\norm{P_N u - u^N}^2_{X_1}\le L(T)N^{1-s}\norm{u}_{X_s}.
\end{equation}
Finally, using~\eqref{eq:1term} and~\eqref{eq:2term} in~\eqref{eq:triangular}, we complete the proof.
\end{proof}

\section{Time discretization}
\label{sec:time}

In this section we consider the full discretization (time discretization) of the semidiscretized system (\ref{perid6}) obtained by applying a quadrature formula to the original problem.
Let us consider the time step size $\tau>0$ and the partition of the time interval $[0,T]$ by means of $t_{n}=n\tau$, for $n=0,\ldots, N_T$,  where $N_T=\left\lfloor \frac{T}{\tau}\right\rfloor$.  Let us  denote $U_n\approx U(t_n)$ and $V_n\approx U'(t_n)$. In what follows, we consider standard time discretization schemes, such as the  St\"ormer-Verlet scheme and the implicit midpoint method, together with less standard procedures based on  a  {\bf   trigonometric} approach.

\subsection{St\"ormer-Verlet scheme}
This is a symplectic, second order in time,  explicit scheme~\cite{verlet}:

\begin{equation}\label{Stormer-verlet0}
 \left\{
\begin{array}{rl}
 V_ {n+\frac{1}{2}}=&V_n+\frac{\tau}{2}[-\Omega^2U_n+ B(t_n)], \\\\
  U_ {n+1}= & U_n+\tau V_ {n+\frac{1}{2}},\\\\
 V_ {n+1}=& V_ {n+\frac{1}{2}}+\frac{\tau}{2}[-\Omega^2 U_{n+1} + B(t_ {n+1})]. 
\end{array}
\right.
\end{equation}
The error, for the time discretization of the St\"ormer-Verlet scheme is well known to be $O(\tau^2)$, while the error in the spatial discretization by the composite  midpoint quadrature is $O(h^2)$; therefore, the overall error of the procedure  (\ref{Stormer-verlet0}) is $O(\tau^2)+O(h^2)$ under sufficient smoothness assumptions on $C$ and $u$. In the case of discontinuities or unboundness of the spatial derivatives of $C$ and/or $u$, the overall error reduces to  $O(\tau^2)+O(h)$.


\subsubsection{von Neumann linear stability of the St\"ormer-Verlet scheme}

Let us consider the  von Neumann analysis to study the stability of the St\"ormer-Verlet scheme  (see  \cite{Morton_Mayers_1998, Lapidus_2003}).
Let us consider the two-step formulation of the scheme applied to the case in which $b(x,t)=0$, that is:
$$  U_ {n+1}-2  U_ {n}+  U_ {n-1}=\tau^2[-\Omega^2U_n].$$
Suppose  to use the midpoint composite formula to approximate the integral in (\ref{perid1}). 
  Let $U_{n,i}$ be the $i$-th component of $U_n$ and reorder the spatial index so that $i$ and $j$  vary  between $-N/2$ and $N/2$ instead of from $0$  to $N$. Then the $i$-th component of the previous equation satisfies:
\begin{equation}\label{Stormer-VerletI}
\rho\ \frac   { U_{n+1,i}-2  U_{n,i}   +U_{n-1,i}  } {\tau^2}= h \sum_{j=-N/2}^{N/2}C_{ij}( U_{n,j}-U_{n,i}  ).
\end{equation}
Let us assume  $U_{n,i}=\mu^n \exp(\phi i\Im) $, $\Im $ the imaginary unit,  $\mu$ is a complex number while  $\phi$ is a positive real number.
We need to determine the conditions on  $\tau$ and $h$ under which $|\mu|\le 1$ (see also   \cite{Silling_Askari_2005}).
Thus, by replacing $U_{n,i}=\mu^n \exp(\phi i\Im ) $ into the numerical scheme   (\ref{Stormer-VerletI}) we obtain:

\begin{equation}\label{Stormer-verletII}
\rho\ \frac   { \mu^{n+1}  -  2\mu^{n}  +\mu^{n-1} 
 } {\tau^2} \exp(\phi i\Im )= h \sum_{j=-N/2}^{N/2}C_{ij} \mu^{n}   [\exp(\phi j\Im)-\exp(\phi i\Im )  ],
\end{equation}
hence,
\begin{equation}\label{Stormer-verletIII}
\rho\ \frac   { \mu -  2  +\mu^{-1} 
 } {\tau^2} = h \sum_{j=-N/2}^{N/2}C_{ij}  [\exp(\phi (j-i)\Im )-1].
\end{equation}
Setting  $q=j-i$, ${\cal C}_q= C_{ij}$ and using the fact that  ${\cal C}_q$ is an even function   (i.e.  ${\cal C}_q={\cal C}_{-q} $)  we have


\begin{equation}\label{Stormer-verletIIII}
\rho\ \frac   { \mu -  2  +\mu^{-1} 
 } {\tau^2} = h \sum_{q=-N'/2}^{N'/2}  {\cal C}_q   [\exp(\phi q \Im )-1]   =2 h \sum_{q=0}^{N'/2}  {\cal C}_q   [\cos(\phi q)-1],
\end{equation}
where $N'$ depends on $i$.


Setting  $\displaystyle \Lambda=  \sum_{q=0}^{N'/2}  {\cal C}_q   [1- \cos(\phi q)]$, then

\begin{equation}\label{verletIII}
\rho\ \frac   { \mu -  2  +\mu^{-1} 
 } {\tau^2}   +2h \Lambda =0 \iff \mu^2 -  2\left (1-  \frac{h \tau^2} {\rho}   \Lambda \right )  \mu  + 1 =0,
\end{equation}
whose roots are


$$ \mu_ {1/2} =(1-  \frac{h \tau^2} {\rho}   \Lambda  )\pm \sqrt  {    \frac{h \tau^2} {\rho} \Lambda \left ( \frac{h \tau^2} {\rho} \Lambda-2 \right)  }.$$
Therefore, the condition such that $|\mu|\le 1$ is given by

$$ \frac{h \tau^2} {\rho} \Lambda-2<0 \iff  \tau<  \sqrt  {  \frac{2 \rho } {h\Lambda} },$$
and since  $\displaystyle \Lambda \le 2  \sum_{q=0}^{N'/2}  {\cal C}_q, $ then
\begin{equation}\label{stab}
\displaystyle \tau<  \sqrt  {  \frac{ \rho } {h\sum_{q=0}^{N'/2}  {\cal C}_q  } } 
\end{equation}
is the condition on $\tau$ and $h$ that should be satisfied in order to have the numerical stability of the scheme.

\subsection{Implicit Midpoint scheme}
This is a symplectic implicit second order scheme: 
\begin{equation}\label{Stormer-verlet}
 \left\{
\begin{array}{rl}
  U_ {n+1}= & U_n+ \frac{\tau}{2} (V_ {n+1}+V_n),   \\\\
 V_ {n+1}=& V_n +\frac{\tau}{2}[-\Omega^2 (U_{n+1}+U_n) + (B(t_n)+B(t_ {n+1}))]. 
\end{array}
\right.
\end{equation}
Such a scheme, being implicit, will allow us to consider larger time step values with respect to the ones used in the explicit formulas. In particular it is linearly unconditionally stable.

\subsection{Trigonometric schemes}

Thanks to the variation-of-constants formula,  the solution in (\ref{perid16}) is 
\begin{equation}\label{Trigonometric_method}
 \left\{
\begin{array}{rl}
 U(t)=&  \cos(t\Omega)U_0+     t\  {\rm sinc}   (t  \Omega) V_ 0+\displaystyle \int_0^t     (t-s)   {\rm sinc}     ( (t-s)  \Omega)  B(s)ds,\\
 V(t)= & -     \Omega  \sin   (t  \Omega) U_ 0  + \cos(t\Omega)V_0  +\displaystyle \int_0^t   \cos   ( (t-s)  \Omega)  B(s)ds,
\end{array}
\right.
\end{equation}
where $\Omega$ is the unique positive definite square root of $\Omega^2$ and sinc$(x)=  \frac{\sin x}{x}$.

A discretization of the variation-of-constants formula~\eqref{Trigonometric_method}  provides the  following explicit  numerical procedure
\begin{equation}\label{Trigm}
 \left\{
\begin{array}{rl}
 U_ {n+1}=&   \cos(\tau\Omega) U_n+ \tau \  {\rm sinc}  (\tau  \Omega) V_ {n}+\displaystyle \int_0^\tau        (\tau-s) \  {\rm sinc} ( (\tau-s)  \Omega)  B(t_n+s)ds,\\\\
V_ {n+1}  = & - \Omega  \sin   (\tau \Omega) U_ n  + \cos(\tau  \Omega)V_n +\displaystyle \int_0^\tau  \cos   ( (\tau-s)  \Omega)  B(t_n+s)ds,
\end{array}
\right.
\end{equation}
enriched  by the initial conditions $U_0$ and $V_0$  [${\rm sinc } (x)  =\frac {\sin  x}{ x}$].  Since we are supposing that $\Omega^2$ is symmetric and definite positive (see Remark \ref{rem1}), then $\Omega$ is the unique positive definite square root of $\Omega^2$. 

When $B$ is constant (i.e. $b(x,t)$ is independent on $t$), this method provides the exact solution at time $t_ {n+1}$; while, in the case of $B$ depending on $t$,  we need to use a quadrature formula to  evaluate the integrals  in  (\ref{Trigm}); in particular we will use a formula with the same accuracy of the one used in the space discretization.

For instance, using the midpoint quadrature  formula 
we derive the following trigonometric scheme of the second order in space and time:
\begin{equation}\label{Trig_meth3}
 \left\{
\begin{array}{rl}
 U_ {n+1}=&   \cos(\tau \Omega) U_n+\tau {\rm sinc }  (\tau  \Omega) V_ {n} +  \displaystyle    \frac{\tau^2}{2}   \   {\rm sinc }  \left(\frac{\tau}{2} \Omega\right)    B\left(t_{ n+\frac{1}{2}}\right), \\ \\
V_ {n+1}  = & -     \Omega  \sin   (\tau \Omega) U_ n  + \cos(\tau  \Omega)V_n + \displaystyle \tau   \cos   \left( \frac{\tau}{2 }  \Omega\right)   B\left(t_{ n+\frac{1}{2}}\right).
\end{array}
\right.
\end{equation}

Instead, using the two-point Gauss quadrature we derive a scheme of the forth order in space and time:
\begin{equation}\label{Trig_meth8}
\resizebox{.99\textwidth}{!}{$
\begin{cases}
 U_ {n+1}=\cos(\tau \Omega) U_n+\tau {\rm sinc }  (\tau \Omega) V_ {n} +
 \displaystyle    \frac{\tau^2}{4}   \left[  \alpha\   {\rm sinc }  \left(      \frac{\tau}{2} \alpha\   \Omega\right)  B\left(t_n   +   \frac{\tau}{2}\beta  \right) +     \beta {\rm sinc }  \left(      \frac{\tau}{2} \beta\   \Omega\right)  B\left(t_n   +   \frac{\tau}{2}\alpha \right)     \right],\\ \\
V_ {n+1}=  - \Omega  \sin   (\tau \Omega) U_ n  + \cos(\tau  \Omega)V_n + 
\displaystyle   \frac{\tau}{2 }   \left[  \cos \left( \frac{\tau}{2 } \alpha    \Omega\right)  B\left(t_n+ \frac{\tau}{2 }\beta  \right)   +   \cos   \left( \frac{\tau}{2 } \beta   \Omega\right)  B\left(t_n+ \frac{\tau}{2 }\beta   \right)  \right],
\end{cases}
$}
\end{equation}
where $\alpha=(1+ \frac  {1} {\sqrt 3}) $ and $\beta=(1- \frac  {1} {\sqrt 3}) $. 
\medskip

Of course the matrices $\Omega$  in (\ref{Trig_meth3}) and (\ref{Trig_meth8}) are different and come respectively from the discretization of the spatial integral  by the midpoint and the two-points Gauss formula.

These schemes require the evaluation of the matrix functions $ \cos(\tau \Omega)$  and ${\rm sinc }  (\tau  \Omega)$, and
while it is possible to compute $\cos(\tau \Omega)$ by using a MATLAB routine,    this  is not possible for ${\rm sinc }  (\tau  \Omega)$. A way to overcome this difficulty is to employ the series expression for  ${\rm sinc }  (\tau  \Omega)$ but this often  results to be expensive and,  more seriously, it can be very inaccurate~\cite{MR2015575}.  If the diagonalization of $\Omega$ is not too expensive then it is better  to first diagonalize $\Omega$  in order to work with  $ \cos(\tau \cdot)$  and ${\rm sinc }  (\tau  \cdot)$ of scalar entries. 

When  $\Omega$ is of  large dimension, the computation of products of functions
 of matrices  (i.e. $ \cos(\tau \Omega)$  and ${\rm sinc }  (\tau  \Omega)$) by vectors could be efficiently done  by means of  Krylov subspace methods  (see for instance \cite{Lopez_Simoncini_2006,  Lopez2006}).   For a review of the computation of the functions $\cos$ and sinc for matrix arguments, the interested reader may refer to~\cite{Hochbruck2008}.

In order to avoid the cost for the inverse of $\Omega$, required in the computation of sinc$(\tau\Omega)$, we can multiply the first row of  (\ref{Trig_meth3}) by $\Omega$ 
\begin{equation}\label{Trig_meth33}
 \left\{
\begin{array}{rl}
 \Omega U_ {n+1}=&   \Omega \cos(\tau \Omega) U_n+ \sin  (\tau  \Omega) V_ {n} +  \displaystyle  \tau  \sin\  (\frac{\tau}{2} \Omega)    B(t_{ n+\frac{1}{2}}    ),   \\ \\
V_ {n+1}  = & -     \Omega  \sin   (\tau \Omega) U_ n  + \cos(\tau  \Omega)V_n + \displaystyle \tau   \cos   ( \frac{\tau}{2 }  \Omega)   B(t_{ n+\frac{1}{2}}    ), 
\end{array}
\right.
\end{equation}
and then solve at each time step a linear system of algebraic equations with the same coefficient  matrix $\Omega$. Similarly, we may reduce the number of flops of  (\ref{Trig_meth8}).


 However, in this case a deep study of the conditioning of $\Omega$ should be done.

\subsubsection{Spectral  linear stability}

Let us consider  the  scalar version of the  problem (\ref{perid6}) with $B(t)=0$, that is
\begin{equation}\label{perid62}
  \begin{pmatrix} u'  \\ v' \end{pmatrix} =     \begin{pmatrix} 0 &1\\ -\omega^2& 0  \end{pmatrix}      \begin{pmatrix}  u \\ v\end{pmatrix},
\end{equation}
where $v=u'$,   the initial conditions are $u_0$ and $v_0$  and $\omega^2$ is the modulus of the largest eigenvalue of $\Omega^2$ .  

If we apply  the St\"ormer-Verlet method  to such  a scalar problem  we derive
   


\begin{equation}\label{peri61}
  \begin{pmatrix}    u_ {n+1} \\  v_ {n+1} \end{pmatrix} =   M(\tau \omega) \begin{pmatrix}    u_ {n}\\    v_ {n}  \end{pmatrix},
\end{equation}
where
$$ M(\tau\omega)=\begin{pmatrix} (1-  \frac{\tau^2}{2} \omega^2)  &\tau\\\\     
      \frac{\tau}{2} (-\omega^2)  (2-  \frac{\tau^2}{2} \omega^2)  &  (1-  \frac{\tau^2}{2} \omega^2)\end{pmatrix}.   $$
The characteristic polynomial of $M(\tau\omega)$ is given by  $\lambda^2- (2-\tau^2 \omega^2)\lambda+1$, thus the eigenvalues of $M(\tau\omega)$ are in modulus equal to 1 if and only if $0<\tau\omega \le 2$, that is  
$$\tau < 2 \sqrt  {  \frac{ \rho } {h k } }, $$
being $\omega^2=hk/\rho$,  where $k$ is the largest eigenvalue of the stiffness matrix $K$. Hence the method results to be conditionally stable and this stability condition should be compared with  (\ref{stab})  obtained by the von Neumann approach.

As far as the linear stability of the implicit midpoint scheme is concerned we have (\ref {peri61}) with 


$$ M(\tau\omega)=  \frac{1} {1+\frac{\tau^2} {4} \omega^2 }    \begin{pmatrix} (1-  \frac{\tau^2}{4} \omega^2)  &\tau\\\\     
     -\tau \omega^2   &  (1-  \frac{\tau^2}{4} \omega^2)\end{pmatrix},$$
\medskip\noindent
whose characteristic polynomial is given by  
$$ p(\lambda)= \frac{1} {1+\frac{\tau^2} {4} \omega^2 }\   [  \lambda^2-   2 (1-  \frac{\tau^2}{4} \omega^2)\lambda + (1-  \frac{\tau^2}{4} \omega^2)^2 +\tau^2\omega^2 ]. $$
Thus, the eigenvalues of $M(\tau\omega)$ are in modulus equal to 1 for each value of $\tau \omega$.  Hence the method results to be unconditionally stable.

If   the trigonometric  method is applied to the  linear scalar problem  we derive (\ref{peri61}) with
   

$$ M(\tau\omega)=\begin{pmatrix} \cos(\tau \omega)      &\tau {\rm sinc }  (\tau  \omega)\\   
     - \omega  \sin   (\tau \omega) &   \cos(\tau \omega) \end{pmatrix},   $$
whose characteristic polynomial is given by  $\lambda^2- 2\cos (\tau\omega) \lambda+1$. Thus, the eigenvalues of $M(\tau\omega)$ are in modulus  equal to 1 for each value of $\tau \omega$, this means that no restriction on $\tau \omega$ will be imposed and the method results to be unconditionally stable. This is also  justified  from the fact that  in this case   the trigonometric method provides the exact solution then no condition on the time step will follow and the only restriction on $\tau$ and  $h$   will be given by accuracy reasons.

\begin {remark}  In the case of autonomous problems  (i.e. B(t)=constant),  the total semidiscretized energy  in (\ref{ene1})  is a quadratic invariant of the  second order differential system  (\ref{perid2}).  The total  discretized energy at $t=t_n$ is given by
\begin{equation}\label{energy}
\mathcal{E}_n=\frac{1}{2} V_n^TV_n+\frac{1}{2} U_n^T\Omega^2 U_n - U_n^TB, \qquad \text{for every $n\ge 0$},
\end{equation}
 and it is well known that  symplectic methods, as  the implicit midpoint method and the St\"ormer-Verlet method,  preserve $\mathcal{E}_n$, that is
$\mathcal{E}_n=\mathcal{E}_0$ (see \cite{hairer2002geometric}). Moreover, even if, the trigonometric methods derived in this paper are not symplectic,  our numerical tests provide a very good  energy preservation, as the numerical tests will show. 
\end{remark}

\section{The nonlinear model of the peridynamics}
\label{sec:nonlin}

In this section we consider the  one-dimensional nonlinear model  (\ref{perid21})  for an homogeneous bar of infinite length
and propose a numerical approach which allows us to use the numerical methods studied for the linear case. Set  $\xi = 
\hat x- x $, and  $ \eta= 
u(\hat x;t)-u(x;t)$. The 
pairwise 
force 
function $f(\xi,\eta)$ 
 may be   considered 
0 outside the interval
horizon $(-\delta,\delta)$.

The general form 
of  a 
pairwise 
force 
function, 
describing  {\bf isotropic}  materials,  is  given by
\begin{equation}\label{nonl3}
f(\xi,\eta)=\phi(|\xi|,|\eta|)\eta.
\end{equation}
An example of such a function   leads to 
the 
so-called    {\bf  bondstretch} model 
\begin{equation}\label{nonl5}
f(\xi,\eta)= c\ s (|\xi|,|\eta|) \ \frac {\eta}{|\eta|},
\end{equation}
where  $c $ is a 
constant (depending 
on 
the 
material 
parameters, 
the 
dimension 
and 
the horizon),  
while 
$$ s (|\xi|,|\eta|)=\frac {|\eta |-|\xi|}{|\xi|},$$
describes 
the  relative change of the Euclidean 
distance of the particles. Notice
that  here the function $f$ is discontinuous in 
its 
first  
argument, and this will reduce the theoretical order of the numerical scheme used. 

 Other
examples 
are 
\begin{equation}\label{nonl6}
f(\xi,\eta)= c\  (|\eta|-|\xi|)^2\ \eta,
\end{equation}
with 
another 
constant 
$c$ (depending 
on 
the 
material 
parameters, 
the 
dimension 
and 
the 
horizon)
and 
\begin{equation}\label{nonl7}
f(\xi,\eta)= a(|\xi|) \  (|\eta|^2-|\xi|^2) \ \eta,
\end{equation}
for 
a 
continuous 
function 
$a$   (depending 
on 
material 
parameters, 
the 
dimension 
and 
the 
horizon) (see for instance  \cite{Emmrich_Puhst_2013, Silling_2000}).

\medskip 
Now, in order to apply the results of the previous section, we linearize the model. Let us assune that $|\eta|<<1$ and that 
 $f(\xi,\eta)$  is  sufficiently smooth. In particular we linearize the function $f(\xi,\cdot)$ with respect to the second variable 
\begin{equation}\label{nonl8}
f(\xi,\eta)\approx f(\xi,0)+C(\xi)\eta 
\end{equation}
where $C(\xi)$ is given by
$$ C(\xi)= \frac  { \partial f(\xi,0)}{\partial \eta}$$
and the term $O(\eta^2)$ has been omitted. Thus, if in  (\ref{neq1})  we replace $f(\xi,\eta)$  with its linear approximation, 
we derive a model of the form (\ref{perid1}). [Usually  $f(\xi,0)=0$, otherwise it can be incorporated into $b$]. In this way the results shown for the linear model hold for the linearized model too, even if, this linearization  will reduce the accuracy of  the theoretical and numerical solution.


 A more accurate method may be derived using the integral form
$$f(\xi,\eta)=f(\xi,0)+\int_0^ {\eta}  \frac  { \partial f(\xi,s)}{\partial \eta}(\eta-s)ds,$$
and then applying an accurate quadrature formula
$$f(\xi,\eta)\approx f(\xi,0)+  \sum_{r=1} ^m w_r   \frac  { \partial f(\xi,s_r)}{\partial \eta}(\eta-s_r),$$
where $w_r$ are the weights while $s_r$ are the nodes of this formula. In general this approach leads to implicit methods,  in fact, if  we use the trapezoidal formula
\begin{equation}\label{trap}
f(\xi,\eta)\approx f(\xi,0)+  \frac  { \eta}{2}  \left[  \frac  { \partial f(\xi,0)}{\partial \eta} +  \frac  { \partial f(\xi,\eta)}{\partial \eta}  \right],  
\end{equation}
we derive a second  order  implicit method.
If $f(\xi,\eta)$ is sufficiently smooth,  an alternative is using a Taylor expansion 
\begin{equation}\label{nonl9}
f(\xi,\eta)\approx f(\xi,0)+C_1(\xi)\eta +\ldots + C_s(\xi)\eta^s,
\end{equation}
where
$$C_i(\xi)= \frac  { \partial^i f(\xi,0)}{\partial \eta^i } \ ,\qquad i=1,\ldots, s,$$
providing an explicit scheme  where higher derivatives  of $f$ with respect to $\eta$ are required.


\section{Numerical tests and simulations}
\label{sec:numerical}
In this section we will provide some numerical simulation to confirm our results. All our codes have been written in {\tt MATLAB} using an Intel(R) Core(TM) i7-5500U CPU @ 2.40GHz computer.

We start with the linear model~\eqref{perid1} with $b(x,t)=0$  where the micromodulus function is given by~\eqref{microm2}. Assume the following initial condition: $ u_{0}(x)=e^{-(x/L)^{2}}$ $x\in\R$ and 
$v=0$, and consider, for simplicity, the parameters $ \rho$, $E$, $l$ and $L$ equal to 1.

The choice of this function is justified by the fact that the decay at infinity makes possible to consider a bounded domain of integration and this approximation improves as $l\rightarrow 0$. 


The theoretical solution for~\eqref{perid1} is~\cite{WECKNER2005705}
\begin{equation}\label{equexact}
u^*(x,t) = \frac{2}{\sqrt{\pi}}\int_0^{\infty}\exp{(-s^2)} \cos\left(2sx\right) \cos\left(2t\sqrt{1-\exp{(-s^2)}}\right)\,ds.
\end{equation} 

We denote by  $  {\bf u}^{\ast}(t)=(u^{\ast}(x_{0},t),...,u^{\ast}(x_{N},t))^{T} $ the theoretical solution vector at the time~$t$ and at the points of the spatial  discretized domain.

Unless otherwise specified, in what follows, we employ the {\tt Mathematica} library to compute the reference solution~\eqref{equexact}. 


\begin{figure}
\centering
\includegraphics[width=0.8\textwidth]{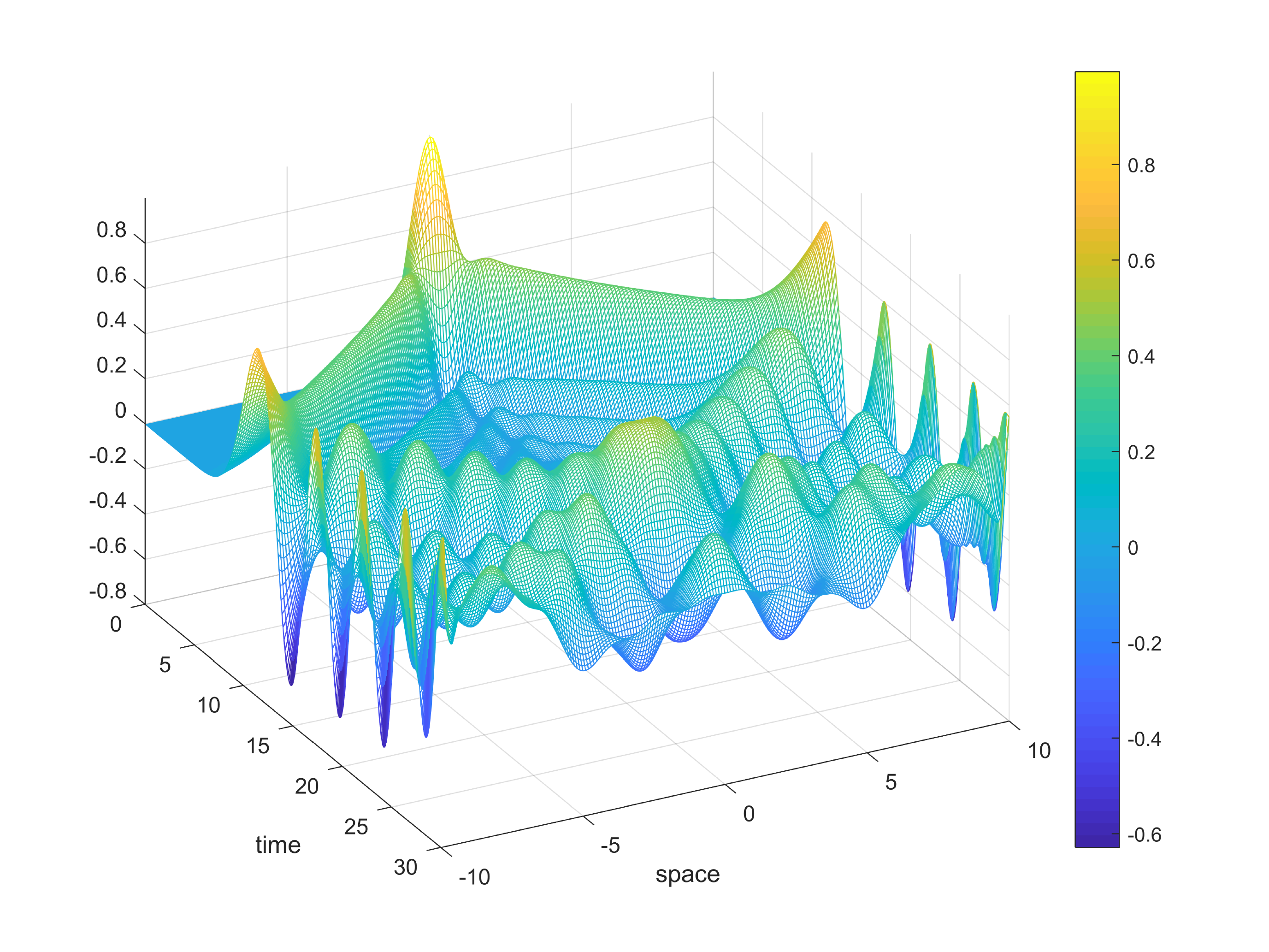} 
\caption{With reference to Test 1: the numerical solution obtained by the MSV method. The parameters for the simulations are $h=\tau=0.1$, $N=200$, $N_T=300$, $\rho=E=l=L=1$.}
\label{fig:sol3d}
\end{figure}
 
To show the errors and the orders of accuracy, we define $ {\bf e}_k $ as
\[  {\bf e}_{k}= \Vert   {\bf u}(t_{k}) -  {\bf  u}^{\ast} (t_{k}) \Vert_{\infty}:=\max \Bigl \{ |u(x_{i},t_{k})-u^{\ast}(x_{i},t_{k})| :  i=0,\ldots, N,\Bigr \}, \]
then, for each method,  we take the maximum error in the time interval $[0,T]$, namely  
$$ || {\bf e} ||_{\infty}:=\max \left \{    {\bf e}_{k} : k=1,\ldots, N_T \right \}. $$

We denote by MT, MSV, MMI and GT the methods consisting of the Midpoint+Trigonometric method, the Midpoint+St\"{o}rmer-Verlet method, the Midpoint+Implicit Midpoint method and the Gauss two points+Trigonometric method, respectively.

\subsection{Test 1: Comparison between MT, MSV, MMI and GT methods}

In this section we study the performance of the MT, MSV, MMI and GT methods by varying the time and space steps. In particular, we compute the error between the exact and the numerical solution and we study the rate of convergence. 

Figure~\ref{fig:sol3d} shows the numerical solution computed by MSV method, while Table~\ref{tab:error-methods} summarizes the errors of the different methods by varying the spatial and time discretization steps.
 In particular, in the MT method we have replaced the matrix $\Omega^{2}$ with the positive definite matrix $\Omega^{2}+h^{\gamma}I$, with $\gamma=2.4$.  
Moreover, for such test, we have assumed that the spatial and time step were equal:  $h=\tau$. Finally, $R_n$ denotes the ratio between the errors corresponding to $h$ and $h/2$, therefore, $\log_2{(R_n)}$ represents the order of convergence of the method.

\begin{table}%
\centering%
\renewcommand\arraystretch{1.3}
\begin{tabular}{|c|c|c|c|c|c|}
\hline
\hline
\rule[-4mm]{0mm}{1cm}
 \textbf{Methods} &  $h=\tau$   & $N$  & $N_T$&\hspace{0.2cm}  $||  {\bf e} ||_{\infty}$ \hspace{0.2cm}&  $\log_2{ (R_n)} $  \\
\hline
\hline

 & $0.100$ & $200$& $30$ &$1.2911\times 10^{-3}$ & - \\

 MSV & $0.050$  &$400$ & $60$ &$3.2340\times10^{-4}$ & $1.9971$ \\

 & $0.025$ &$800$ & $120$ &$8.0821\times10^{-5}$ & $2.0004$\\
\hline
 & $0.100 $ &  $200$& $30$ &$5.9276\times10^{-3}$ & - \\
 
MT & $0.050$  &$400$ & $60$ &$1.1126\times10^{-3}$ & $2.3959$\\
 
 & $0.025$  &$800$ & $120$ &$2.1350\times10^{-4}$& $2.3992$\\
\hline
& $0.100 $ &$200$ & $30$ &$2.5754\times10^{-3}$& -  \\
 
 MMI & $0.050$  & $400$& $60$ &$6.4621\times10^{-4}$ & $1.9946$\\
 
 & $0.025$ & $800$& $120$ &$1.6106\times10^{-4}$& $2.0043$\\

\hline
& $0.100 $ &$400$ &$30$  &$1.4940\times10^{-4}$ & -  \\
 
GT & $0.050$  &$800$ &$60$  &$9.3380\times10^{-6}$ & $3.9998$\\
 
 & $0.025$ &$1600$ &$120$  &$5.8300\times10^{-7}$& $4.0015$\\
\hline
\hline
\end{tabular}
\renewcommand\arraystretch{1}
\caption{With reference to Test 1: the comparison among MSV, MT, MMI and GT methods by varying $h$, $\tau$, $N$ and $N_T$. The parameters for the simulation are $\rho = E= l = L = 1$.}
\label{tab:error-methods}
\end{table}
Looking at $\log_2{(R_n)}$, in the last column of Table~\ref{tab:error-methods}, we see that the methods MSV, MT, MMI are of the second order of accuracy while GT is of the fourth order, but GT is more expensive because it uses a double number of nodes compared with MT and the evaluation of functions of matrices. 
 The method MSV is computationally less expensive than the others, but it has a bounded stability region, see Table~\ref{tab:stability-region}  where we have placed the Young's modulus  $E=100$.

\begin{table}%
\centering%
\renewcommand\arraystretch{1.3}
\begin{tabular}{|c|c|c|c|c|c|}
\hline
\hline
\rule[-4mm]{0mm}{1cm}
\textbf{Methods}  & $h$ &  $\tau$  & $N$ & $N_T$ &  \hspace{0.2cm}  $||  {\bf e}  ||_{\infty}$ \hspace{0.2cm} \\

\hline
\hline
 &$0.100$&  $0.100$ &$200$ &$300$ & $1.0543$  \\

MSV &$0.050$ & $0.200$ & $400$ & $150$ &  $2.6300\times10^{168}$\\

& $0.025$ & $0.400$ & $800$ & $75$&  $4.3600\times10^{131}$ \\

\hline

 &$0.100$&  $0.100$ &$200$ &$300$ & $1.0941$ \\ 

MT&$0.050$ & $0.200$ & $400$ & $150$ &  $1.1081$ \\

& $0.025$ & $0.400$ & $800$ & $75$ & $1.2987$  \\
\hline

 &$0.100$&  $0.100$ &$200$ &$300$ & $1.0923$  \\
MMI&$0.050$ & $0.200$ & $400$ & $150$ &   $1.0925$ \\
& $0.025$ & $0.400$ & $800$ & $75$ &  $8.2060\times10^{-1}$ \\
\hline
\hline
\end{tabular}
\renewcommand\arraystretch{1}
\caption{With reference to Test 1: the maximum error for the methods MSV, MT and MMI for different choices of $h$, $\tau$, $N$ and $N_T$. The parameters for the simulation are $\rho = l=L=1$, $E=100$.}
\label{tab:stability-region}
\end{table}

\subsection{Test 2: The conservation of the total semidiscretized energy in the autonomous case}
As far as the conservation of the energy of the semidiscretized problem is concerned,   we should have that  $\mathcal{E}_n-\mathcal{E}_0=0$, see~\eqref{energy}, and  
in Figure~\ref{fig:energy} we show the comparison between the energy conservation obtained by the MSV and MT methods in the time interval $[0,30]$ and for a number of spatial nodes equal to 200. We observe that the  maximum variation of the numerical energy is of order $10^{-2}$.
\begin{figure}
\centering            
\includegraphics[width=0.6\textwidth]{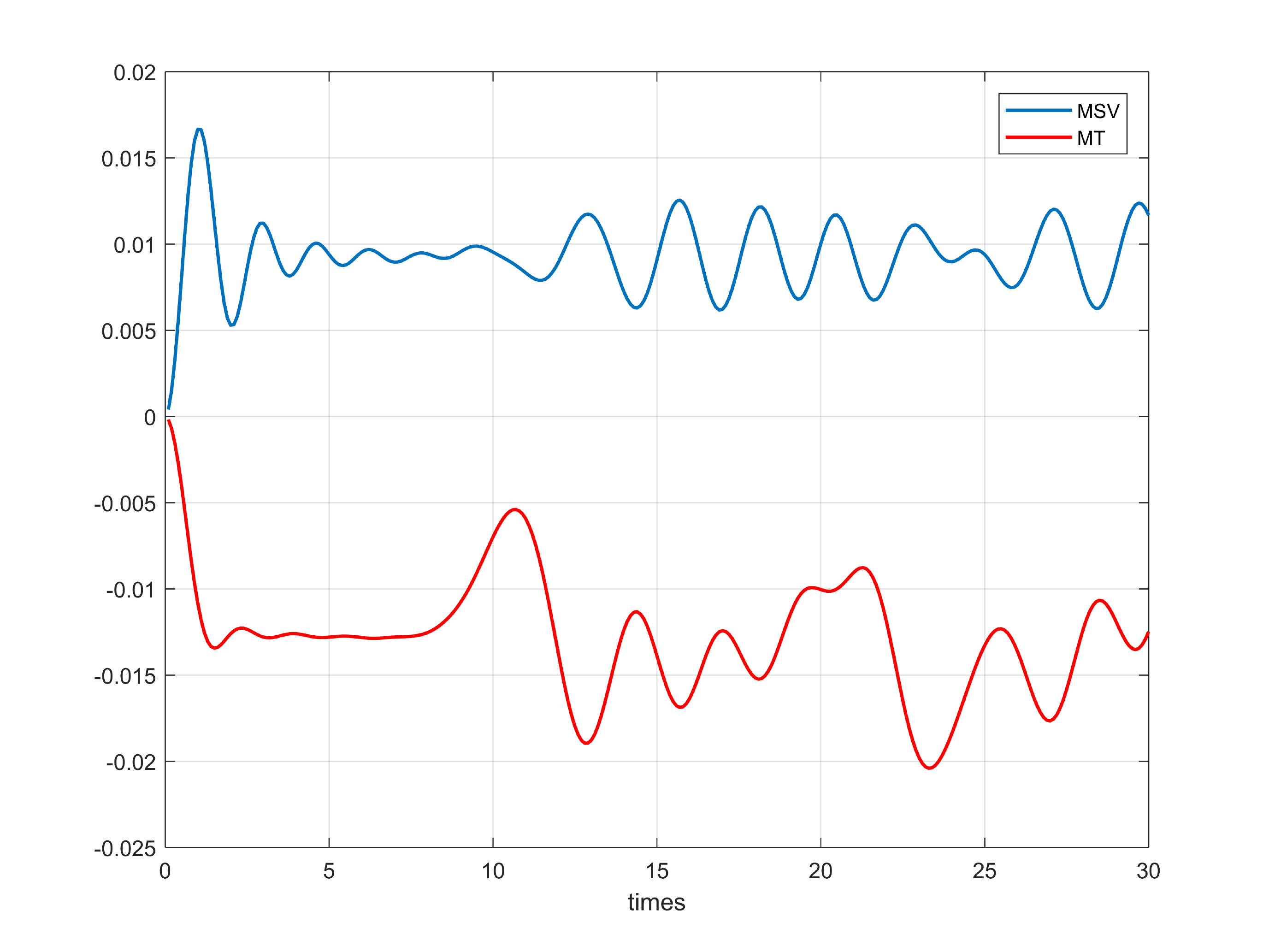} 
\caption{With reference to Test 2: the energy variation $\mathcal{E}_n-\mathcal{E}_0$ associated with MSV and MT methods for $N= 200$.}
\label{fig:energy}
\end{figure}
If  we double the number of spatial nodes to 400, the maximum variation of the energy is of order $10^{-3}$ showing that $\mathcal{E}_n$ depends also on the error of the quadrature formula used to discretize the spatial domain.  




\subsection{Test 3: A comparison between the numerical solution of the linear peridynamic equation with the solution of the wave equation}

We now compare the numerical solution  of the linear peridynamic equation with the solution  of the wave equation in \eqref{wave1}. We define the difference vector
\[   {\bf  d}_{k}= \Vert    {\bf  u} ^{\ast}  (t_{k}) -  {\bf u}^{\ast \ast} (t_{k}) \Vert_{\infty},\quad\text{for }k=1,...,n,\]
where $ {\bf u} ^{\ast }   (t)=(u(x_{0},t),...,u(x_{N},t))^{T}$ is  the numerical solution at the spatial points of the peridynamic equation, while $ {\bf u}^{\ast \ast}   (t)=(u(x_{0},t),...,u(x_{N},t))^{T}$  is the numerical solution at the spatial points of the wave equation.


In Table~\ref{tab:wave-comparison}, we have reported the maximum difference between ${\bf u} ^{\ast }   (t)$ and ${\bf u}^{\ast \ast}   (t)$ as $l$ goes to zero.
\begin{table}
\centering%
\renewcommand\arraystretch{1.3}
\begin{tabular}{|c|c|c|}
\hline
\hline
\rule[-4mm]{0mm}{1cm}
\textbf{Methods} & $l/L$ & $ \hspace{0.4cm} ||    {\bf d}||_{\infty}  $ \hspace{0.4cm} \\
\hline
\hline
 & $0.400$  & $5.4948\times10^{-2}$ \\
MSV & $0.200$  & $1.2269\times10^{-2}$ \\
 & $0.100$ &  $2.4625\times10^{-3}$ \\
\hline
 & $0.400$  &  $5.2569\times10^{-2}$ \\
MT & $0.200$  & $1.5168\times10^{-2}$ \\
 & $0.100$ &  $6.0420\times10^{-3}$ \\
\hline
 & $0.400$  &  $5.6887\times10^{-2}$  \\
GT & $0.200$  & $1.4646\times10^{-2}$ \\
 & $0.100$ &  $3.7111\times10^{-3}$  \\
\hline
 & $0.400$  &  $6.0951\times10^{-2}$  \\
MMI & $0.200$  &    $1.9493\times10^{-2}$  \\
 & $0.100$ &     $9.6978\times10^{-3}$ \\
\hline
\hline
\end{tabular}
\renewcommand\arraystretch{1}
\caption{With reference to Test 3: the maximum distance between ${\bf u} ^{\ast }   (t)$ and ${\bf u}^{\ast \ast}   (t)$ as function of the ratio $l/L$ for different methods.}
\label{tab:wave-comparison}
\end{table}

\subsection{Test 4: Validation of spectral semi-discretization scheme}
In this section we implement and validate the scheme proposed in Section~\ref{sec:spectral}. We consider the linear model~\eqref{perid1} and we take the micromodulus function $C(x) = \frac{4} {\sqrt{\pi}}\, \exp{(-x^2)}$, as in~\eqref{microm2}, where for simplicity we take $E=l=1$. We assume that the body is not subject to external forces, namely $b(x,t) \equiv 0$ and the density of the body is $\rho(x) = 1$. As initial condition, we choose $u_0(x) = \exp{(-x^2})$ and $v(x) =0$.

We denote by $u^*(x,t)$ the reference solution for the problem given by~\eqref{equexact}. 
Since $u^*(x,t)$ decays exponentially to zero as $|x|\to\infty$, we can truncate the infinite interval to a finite one $[-M\pi, M\pi]$, with $M>0$, and we approximate the boundary conditions by the periodic boundary conditions on $[-M\pi, M\pi]$. It is expected that the initial-boundary valued problem can provide a good approximation to the original initial-valued problem as long as the solution does not reach the boundaries.

Notice that, in this simple case, we do not need to use a time discretization for solving~\eqref{eq:ode}. Indeed, we have
\[
\omega_k^2 = \frac{8}{\sqrt{\pi}}\int_0^\infty \exp{(-\xi^2)}\left(1-\cos(k\xi)\right)\,d\xi= 4\left(1-\exp{(-\frac{k^2}{4})}\right), 
\]
hence, the solution of the homogeneous Cauchy problem~\eqref{eq:ode} in the frequencies space
is 
\[
\tilde{u}_k(t) = \tilde{u}_{0,k}\cos\left(\omega_k\,t\right).
\]
We fix a constant space step $h = 10^{-3}$, $M=2.5$ and we set $N = 2\left\lfloor\frac{\pi}{h}\right\rfloor = 6284$. Figure~\ref{fig:comparison} shows the comparison between the exact solution and its numerical approximation at different times. 

\begin{figure}
\centering
\begin{subfigure}[b]{.32\textwidth}
\includegraphics[width=\textwidth]{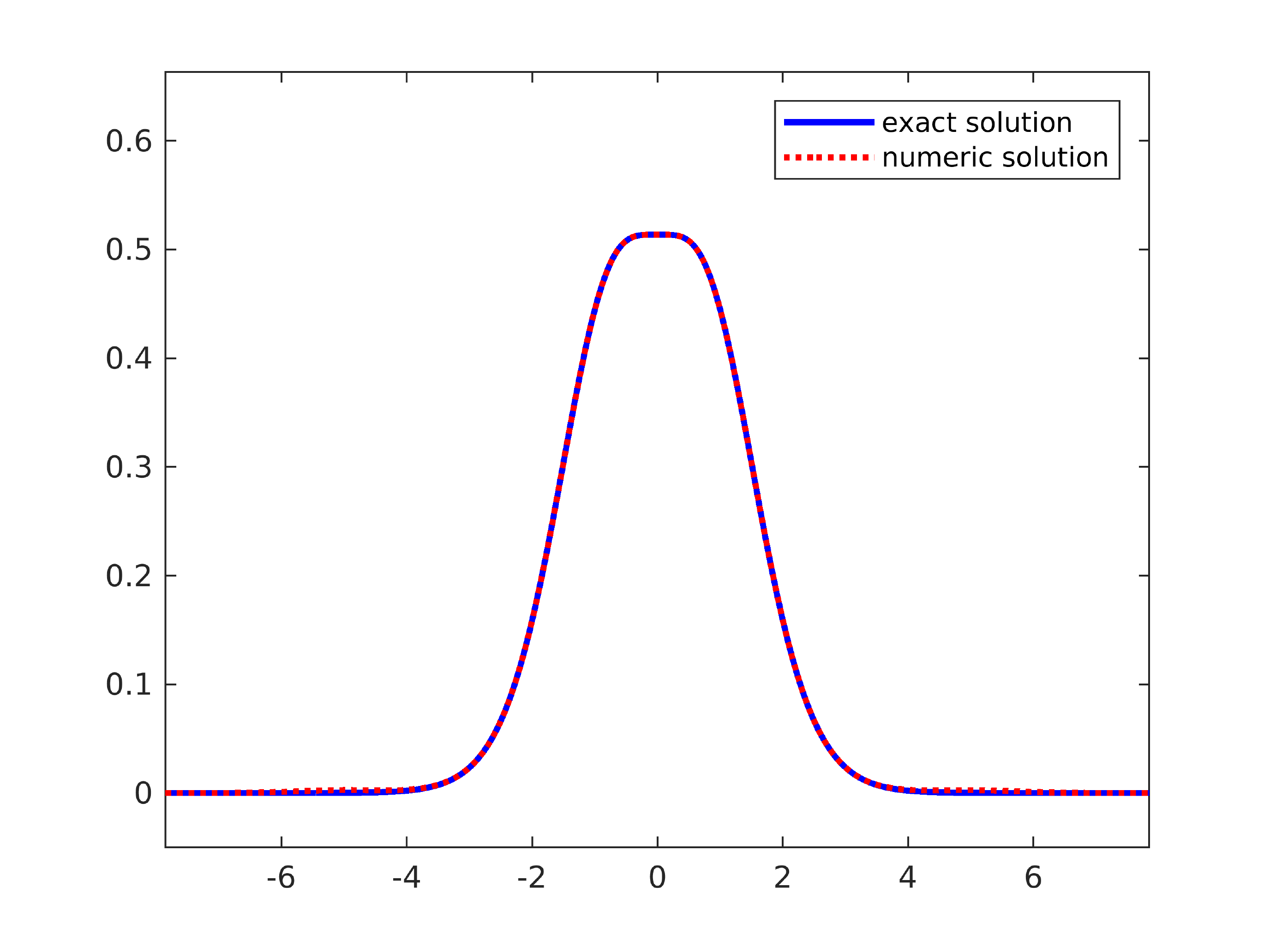}
\caption{$t=1$.}
\end{subfigure}
\,
\begin{subfigure}[b]{.32\textwidth}
\includegraphics[width=\textwidth]{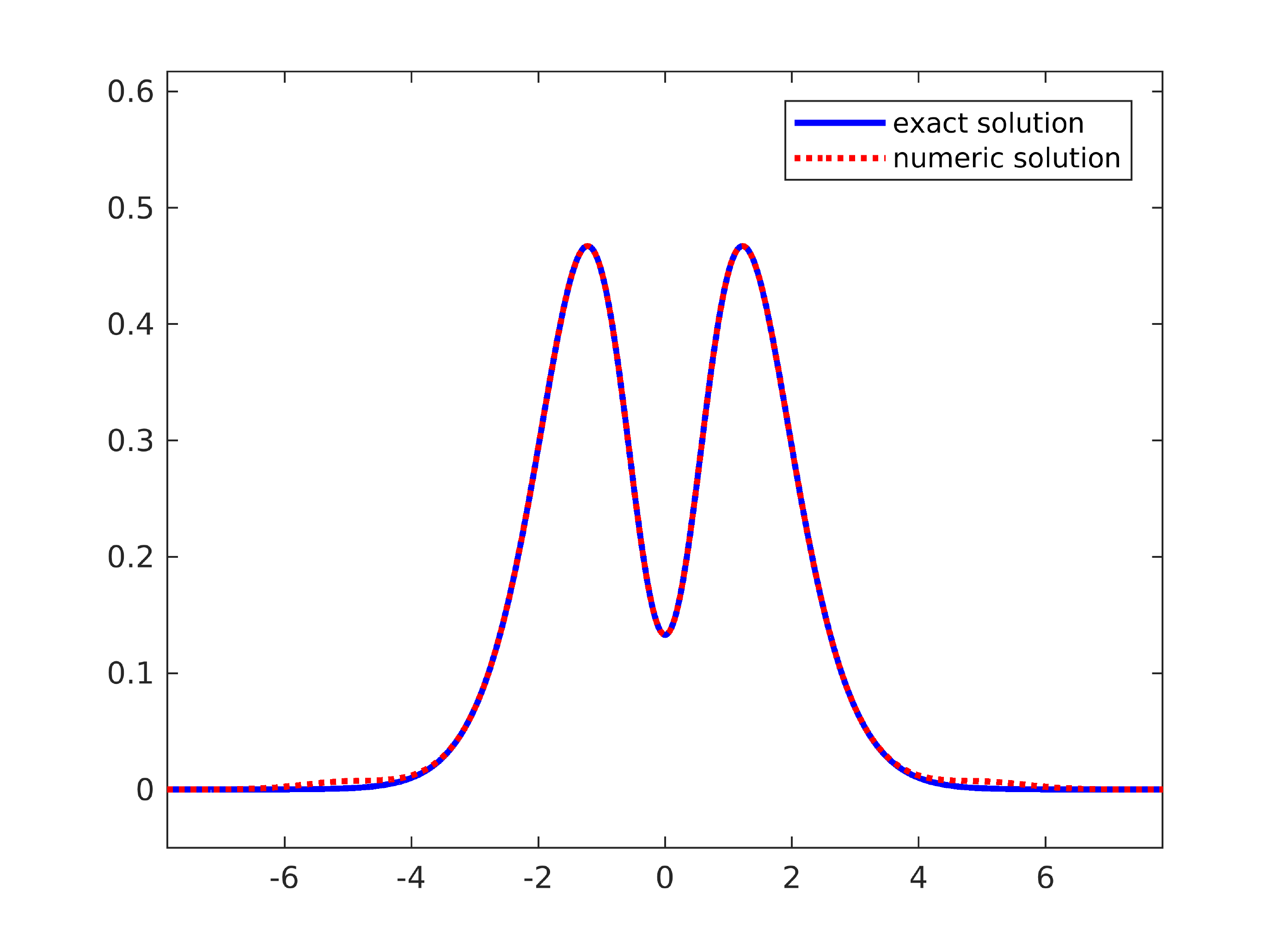}
\caption{$t=1.5$.}
\end{subfigure}
\,
\begin{subfigure}[b]{.32\textwidth}
\includegraphics[width=\textwidth]{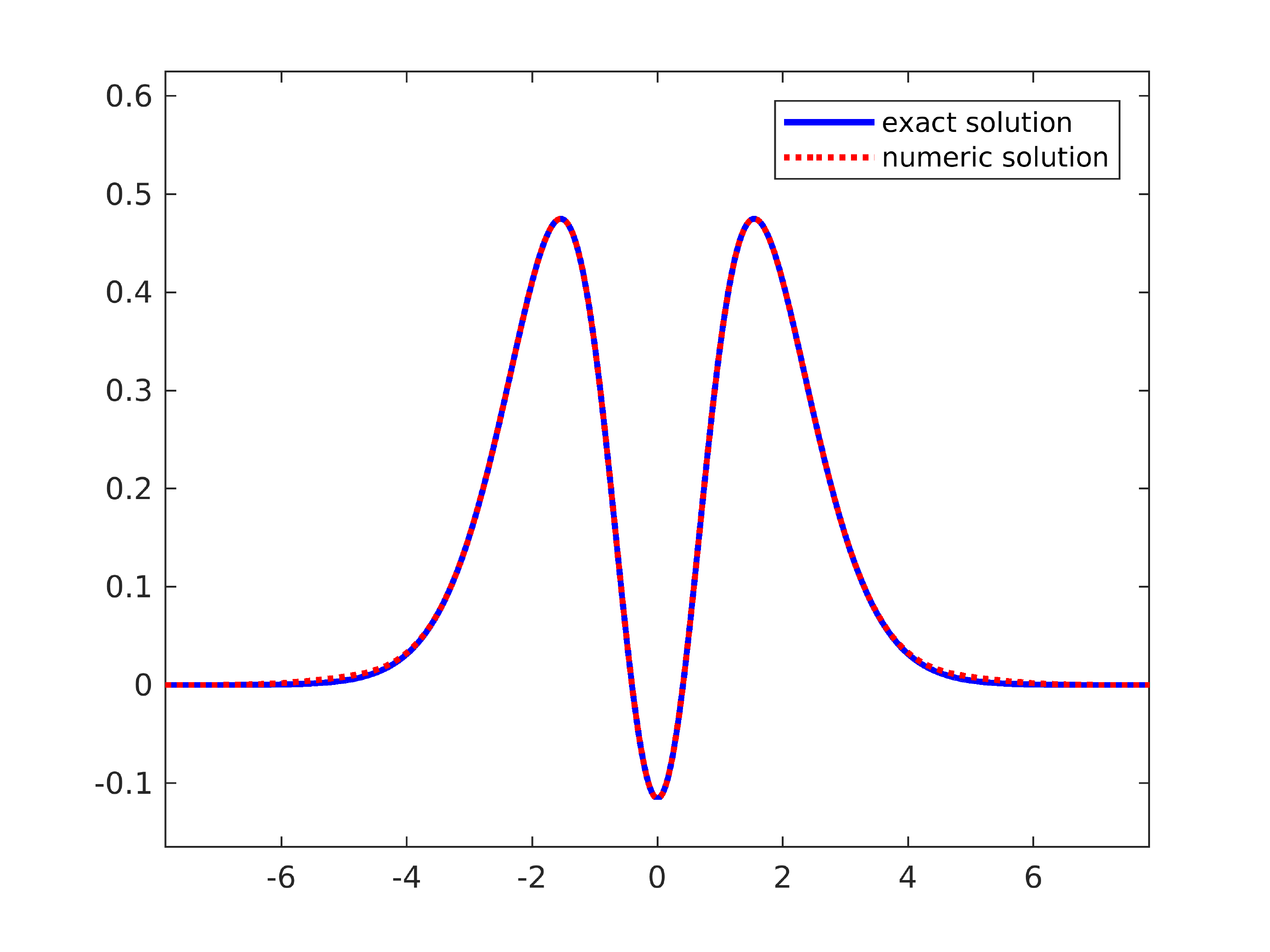}
\caption{$t=2$.}
\end{subfigure}
\\
\begin{subfigure}[b]{.32\textwidth}
\includegraphics[width=\textwidth]{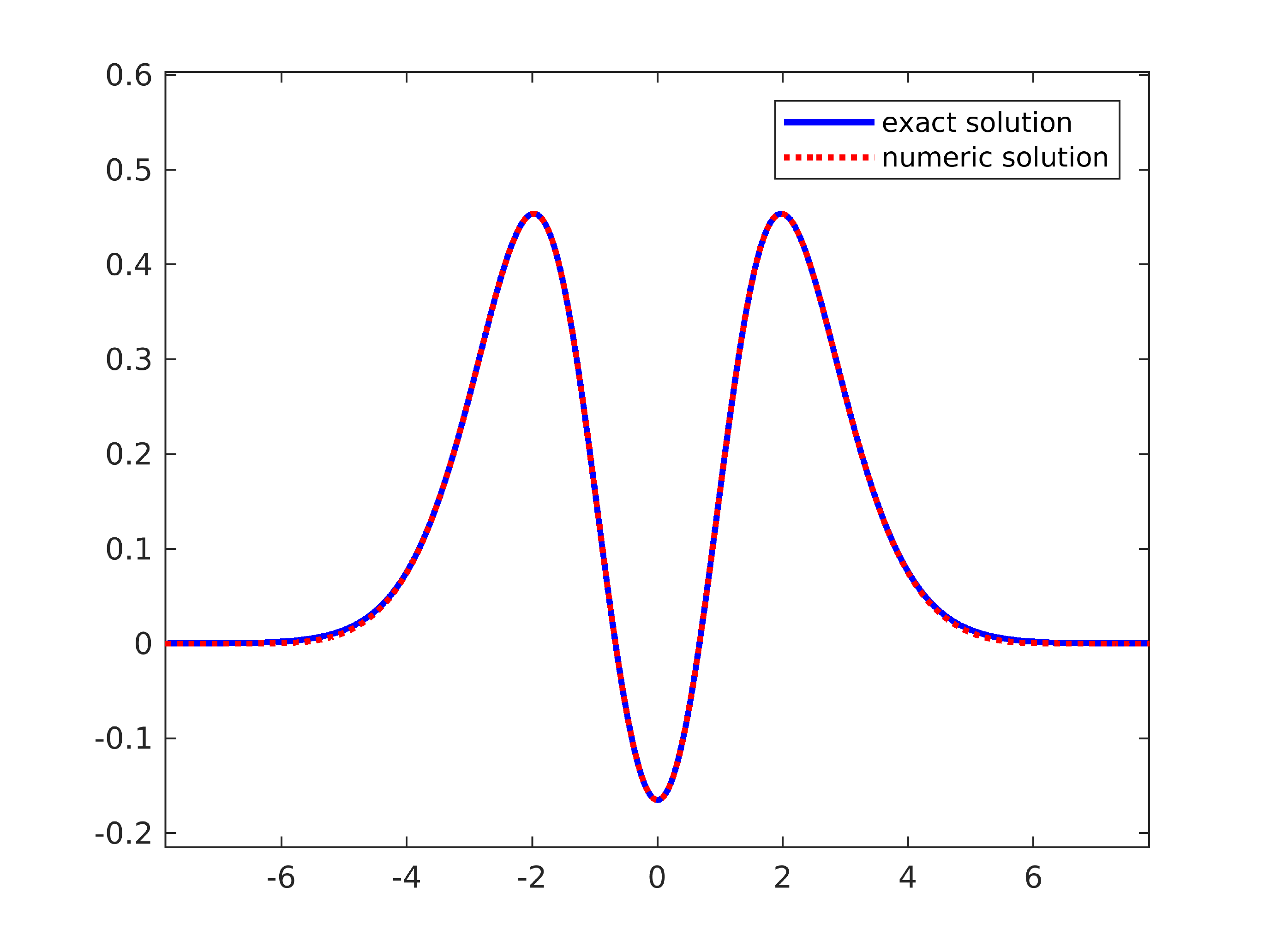}
\caption{$t=2.5$.}
\end{subfigure}
\,
\begin{subfigure}[b]{.32\textwidth}
\includegraphics[width=\textwidth]{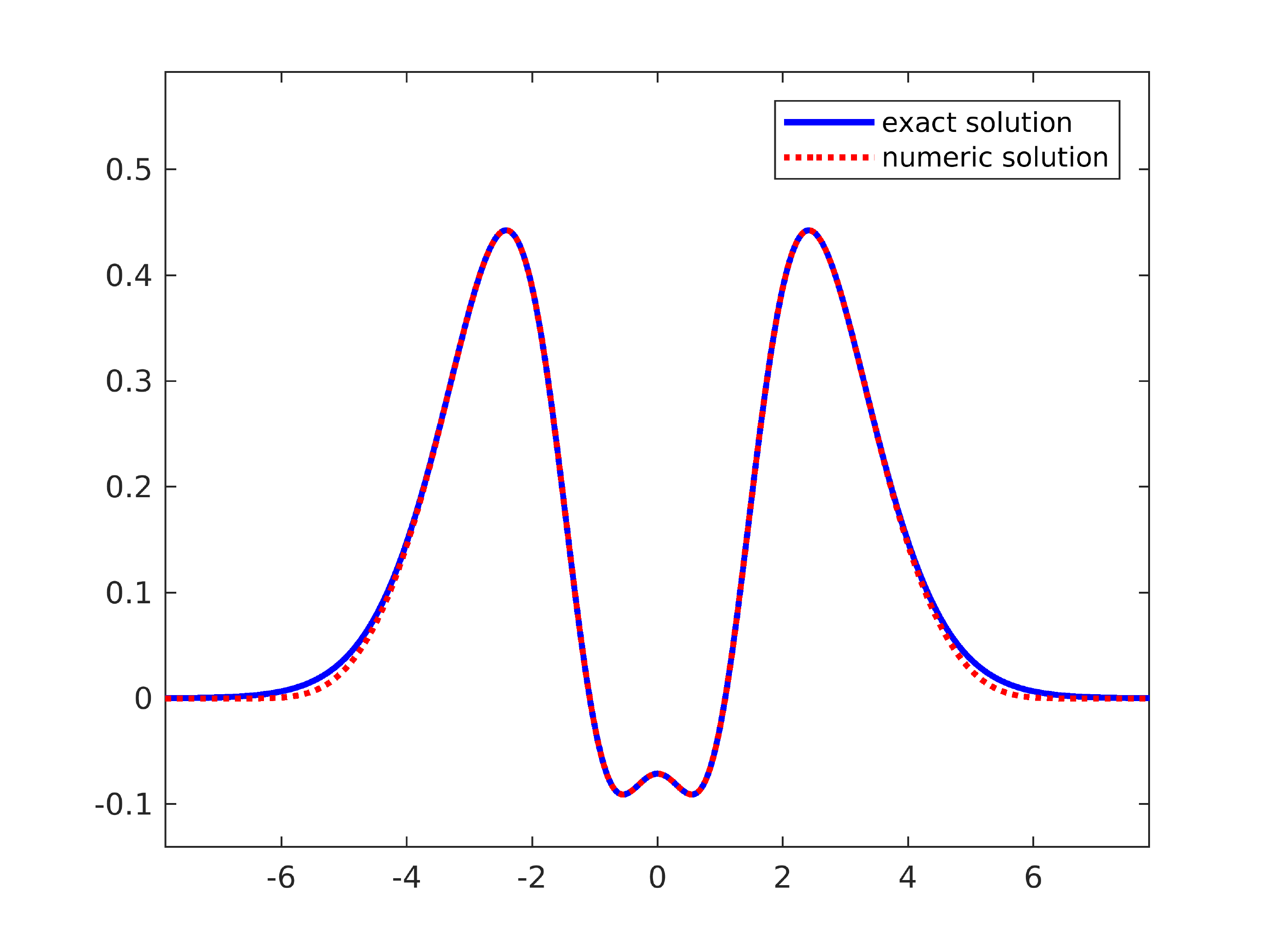}
\caption{$t=3$.}
\end{subfigure}
\,
\begin{subfigure}[b]{.32\textwidth}
\includegraphics[width=\textwidth]{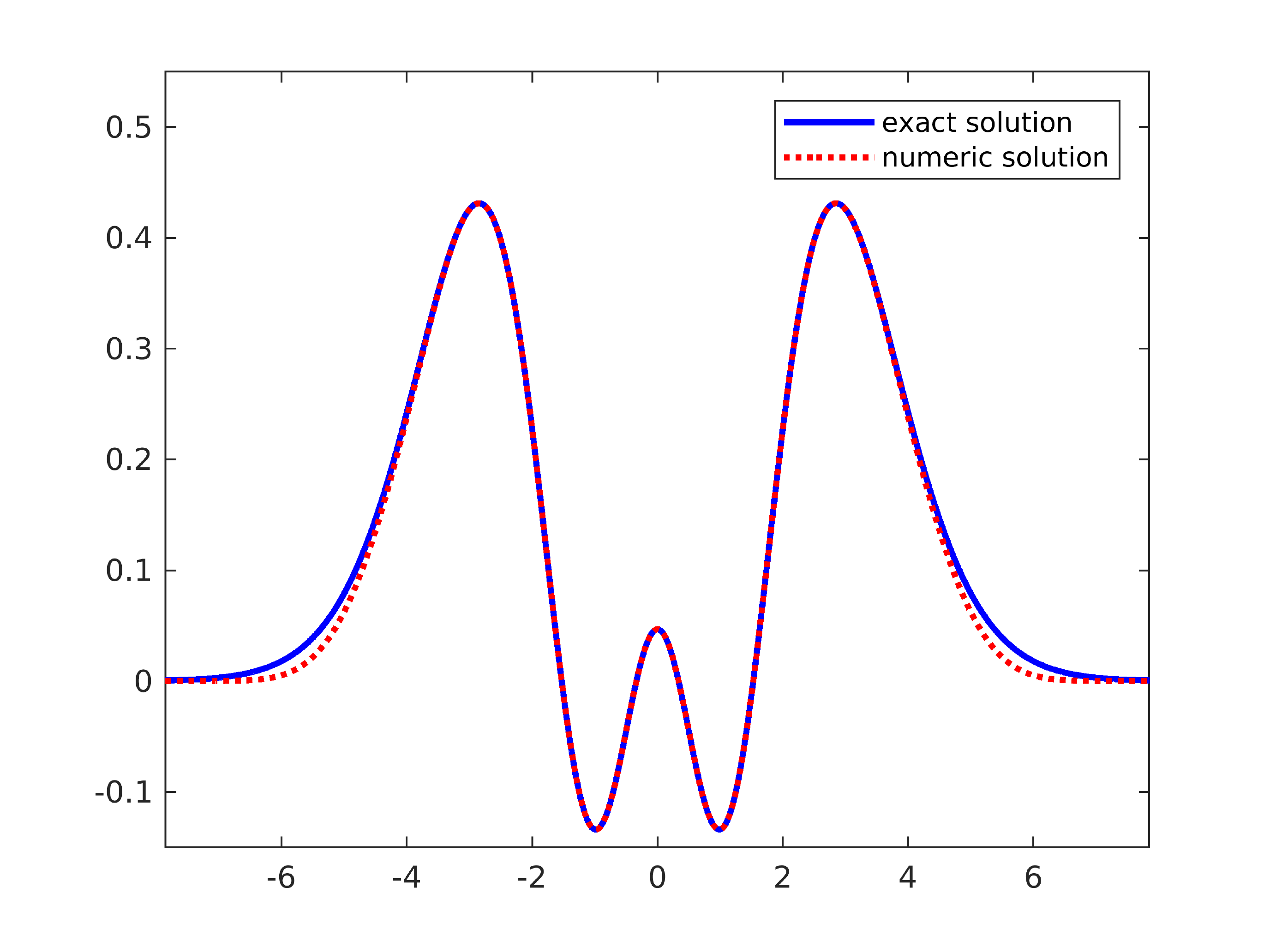}
\caption{$t=3.5$.}
\end{subfigure}
\caption{With reference to Test 4: the comparison between exact and approximated solution at six different times. The parameters for the simulation are $E=l=\rho=1$, $h=10^{-3}$, $M=2.5$, $N=6284$.}
\label{fig:comparison}
\end{figure}

In Figure~\ref{fig:error} we plot respectively the distance and the square distance between the exact solution and its numerical approximation for various $N$ using the semilogy scale. The appearance of ``spikes'' in the error approaching zero confirms the interpolating nature of the spectral operator. Observe that the error grows as we approach the boundaries. This is a typical phenomenon when dealing with spectral methods. More precisely, such aspect occurs whenever one approximate an initial-valued problem with an initial-boundary valued problem with periodic boundary conditions. Therefore, in order to avoid such aspect and to perform an error study, we restrict our attention to a suitable subinterval of the domain. For simplicity, we work on the interval $[-\pi,\pi]$.

\begin{figure}
\centering
\begin{subfigure}[b]{.45\textwidth}
\includegraphics[width=\textwidth]{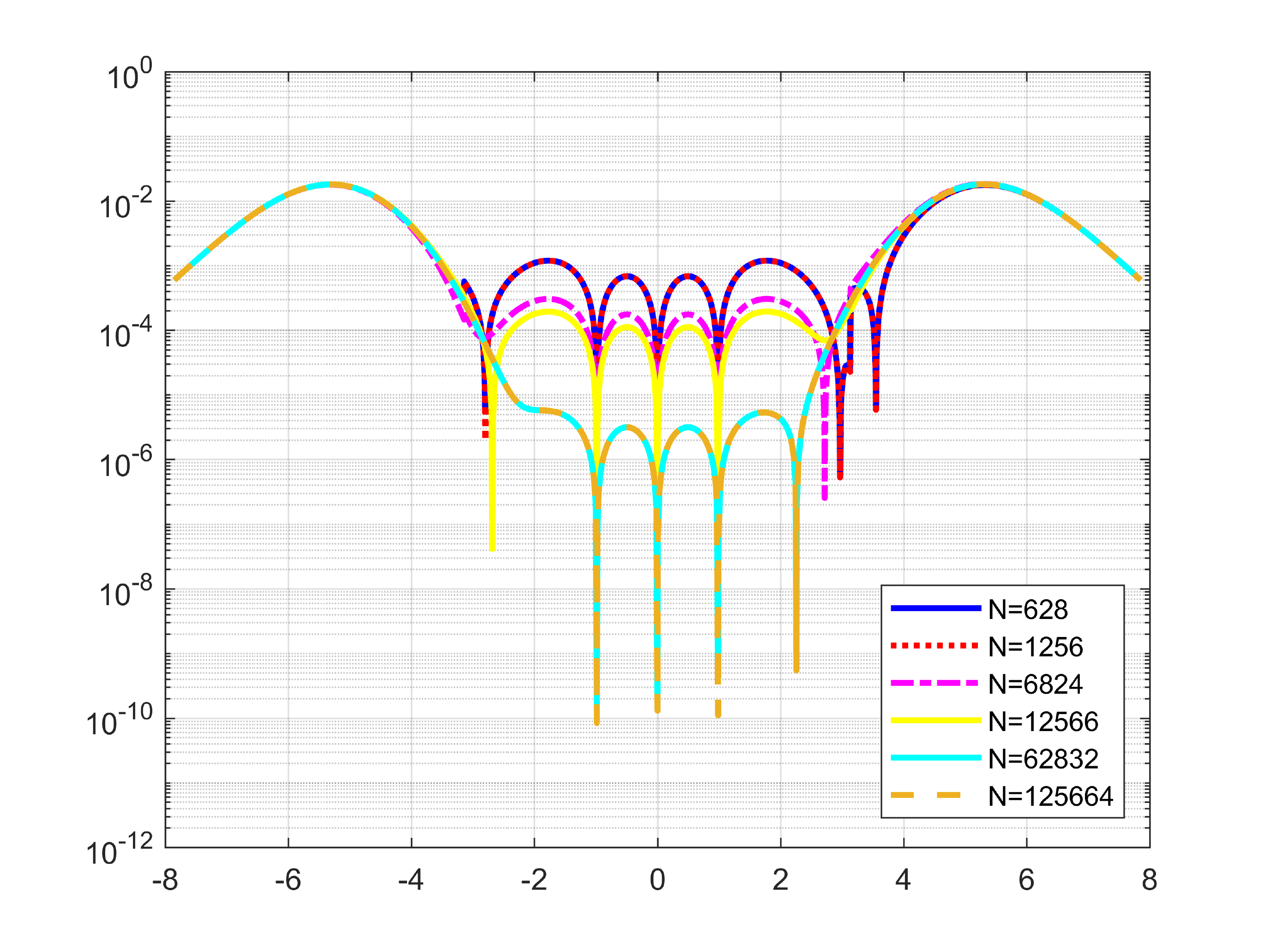}
\caption{$\left|u_N(x,3.5)-u^*(x,3.5)\right|$.}
\end{subfigure}
\,
\begin{subfigure}[b]{.45\textwidth}
\includegraphics[width=\textwidth]{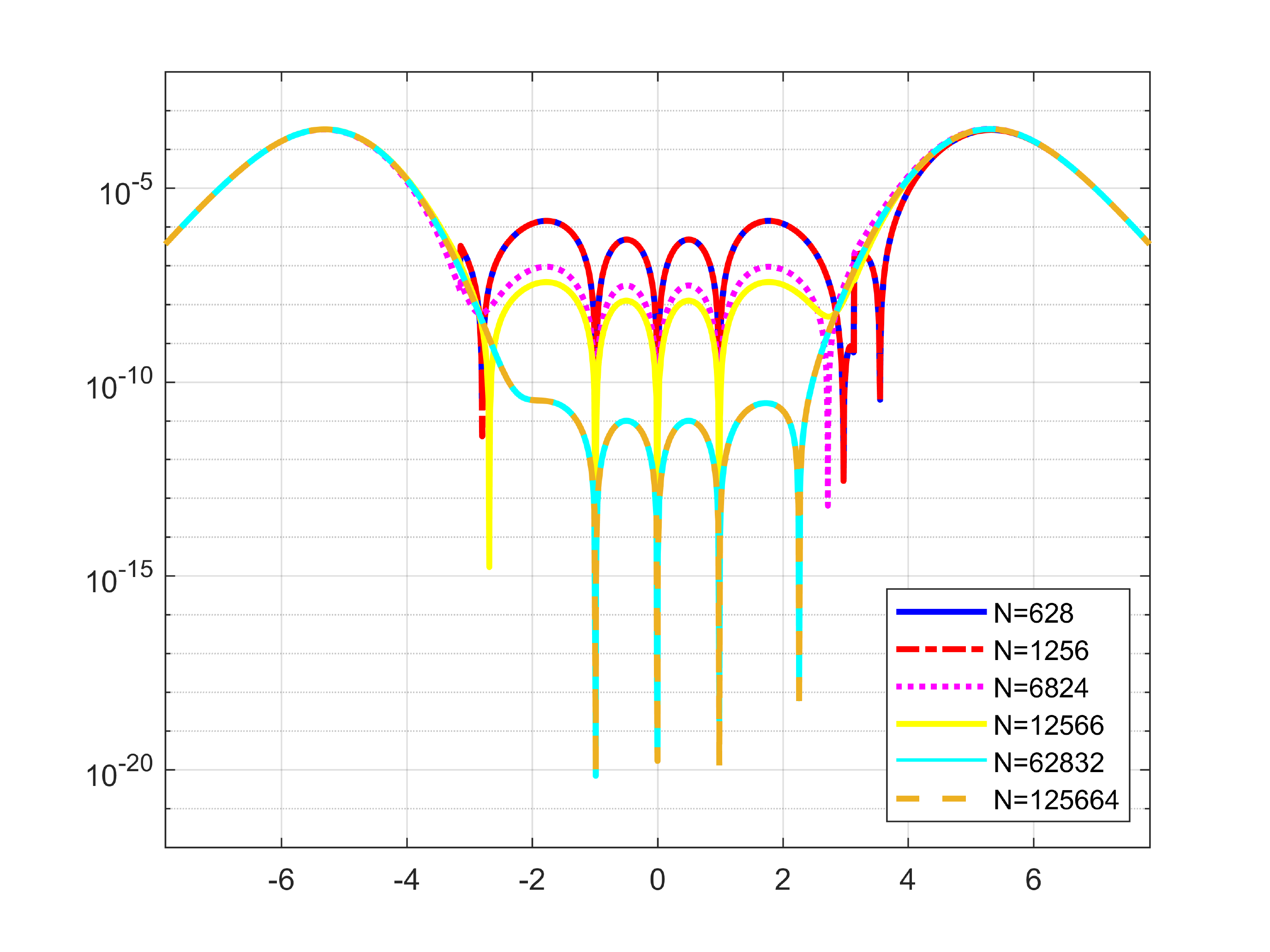}
\caption{$\left|u_N(x,3.5)-u^*(x,3.5)\right|^2$.}
\end{subfigure}
\caption{With reference to Test 4: the error for various $N$ using the semilogy scale. The parameters of the simulation are $E=l=\rho=1$, $h=10^{-3}$, and $M=2.5$.}
\label{fig:error}
\end{figure}


We perform an error study for this test in $[-\pi,\pi]$: we introduce the relative pointwise-error and the relative $L^2$-error respectively as follows
\[
E^t_{L^\infty} = \frac{\max_{j} \left|u_N(x_j,t)- u^*(x_j,t)\right|}{\max_{j} \left|u_N(x_j,t)\right|},\qquad E^t_{L^2} = \frac{\sum_{j}\left|u_N(x_j,t)- u^*(x_j,t)\right|^2}{\sum_{j}\left|u_N(x_j,t)\right|^2}.
\]

Table~\ref{tab:error} and Figure~\ref{fig:norm-comparison} depict the relative pointwise error and the relative $L^2$-error for increasing resolution at the fixed time $t=3.5$. 
\begin{figure}
\centering
\begin{subfigure}[b]{.45\textwidth}
\includegraphics[width=\textwidth]{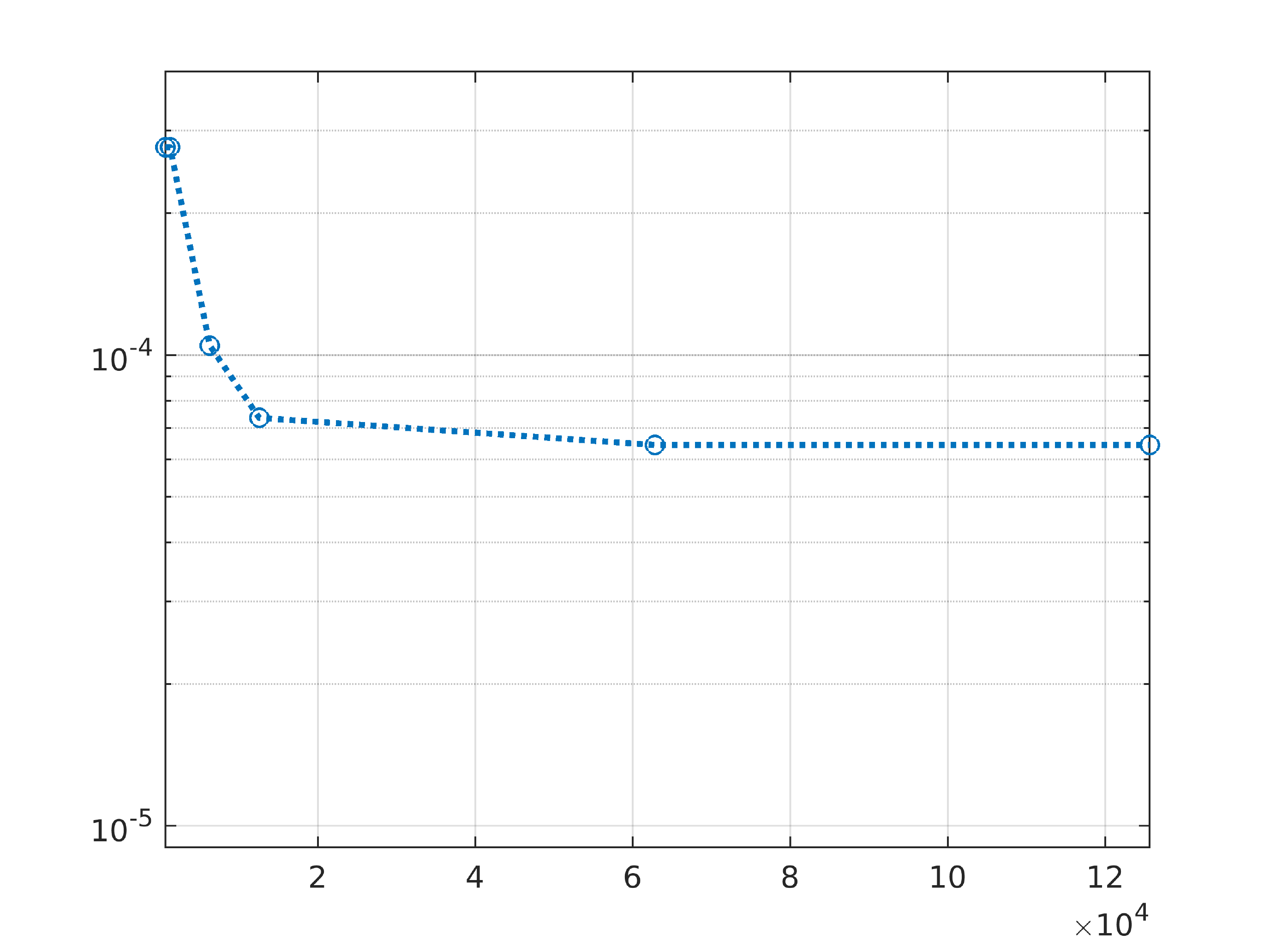}
\caption{$E^t_{L^\infty}$ by varying $N$.}
\end{subfigure}
\,
\begin{subfigure}[b]{.45\textwidth}
\includegraphics[width=\textwidth]{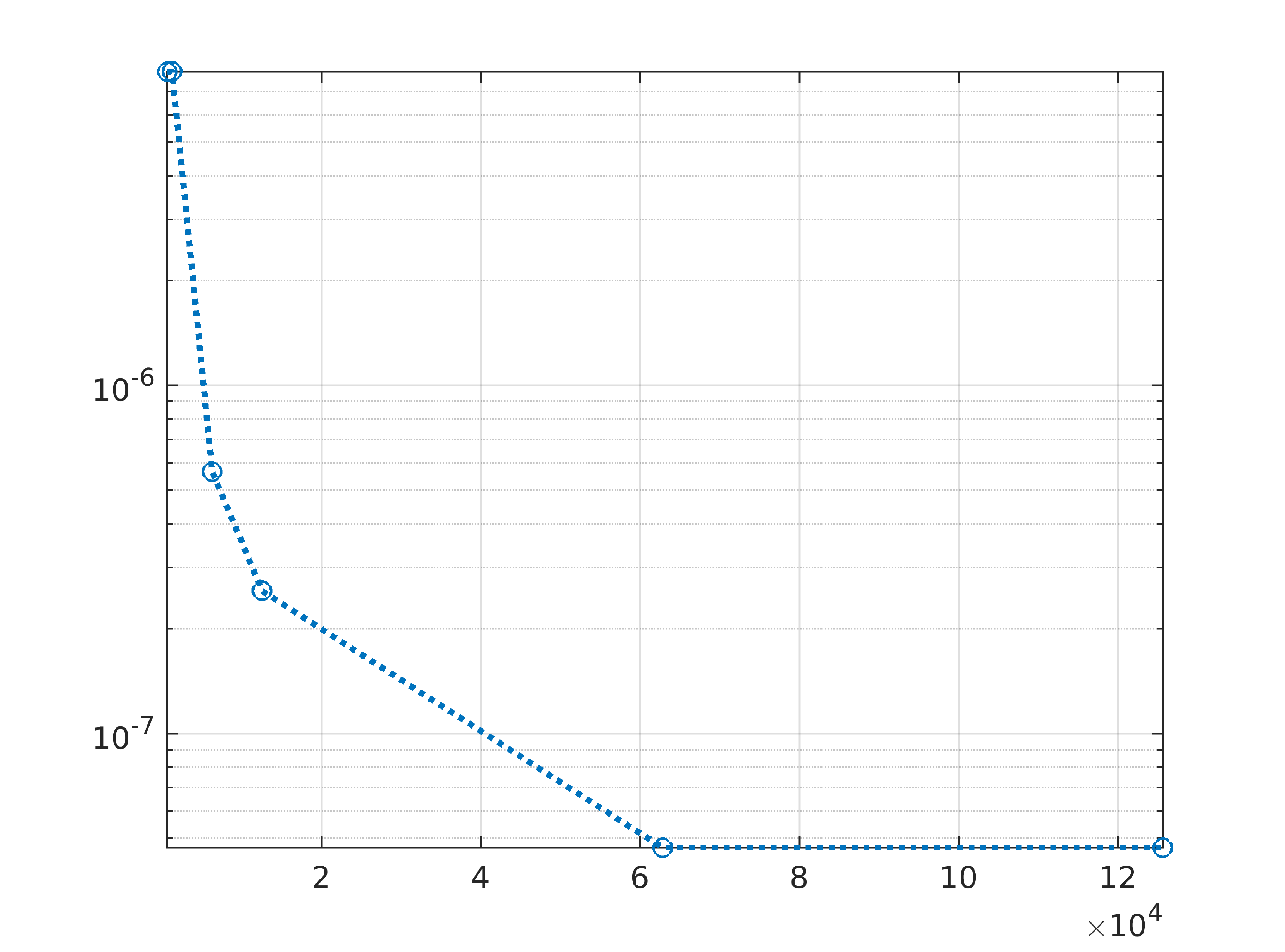}
\caption{$E^t_{L^2}$ by varying $N$.}
\end{subfigure}
\caption{With reference to Test 4: the comparison between the errors by varying $N$, using the semilogy scale. The parameters for the simulation are $h=10^{-3}$, $t=3.5$, $M=1$, $E=l=L=\rho=1$.}
\label{fig:norm-comparison}
\end{figure}

\begin{table}%
\centering%
\renewcommand\arraystretch{1.3}
\begin{tabular}{|c|c|c|
}
\hline
\hline
N&  $E_{L^\infty}^{t}$ & $E_{L^2}^{t}$ 
\\
\hline
\hline
$628$&$2.7628 \times 10^{-4}$&$7.9603\times10^{-6}$
\\
\hline
$1256$&$2.7628\times 10^{-4}$&$7.9774\times10^{-6}$
\\
\hline
$6284$&$1.0474\times 10^{-4}$&$5.6593\times10^{-7}$
\\
\hline
$12566$&$7.3552\times10^{-5}$&$2.5697\times10^{-7}$
\\
\hline
$62832$&$6.4412\times 10^{-5}$&$4.7057\times 10^{-8}$
\\
\hline
$125664$&$6.4412\times 10^{-5}$&$4.7048\times 10^{-8}$
\\
\hline
\hline
\end{tabular}
\renewcommand\arraystretch{1}
\caption{With reference to Test 4: the relative pointwise-error and relative $L^2$-error at time $t=3.5$ for different values of $N$ in the computational domain $[-\pi,\pi]$.}
\label{tab:error}
\end{table}

\subsection{Test 5: Comparison between MSV and MMI in the nonlinear case}

We now consider the case in which the pairwise force function is non linear with a finite horizon $\delta>0$. In particular, we will deal with the model  in which $f$  has the following form
\begin{equation*}
f(\xi,\eta)=\begin{cases}
c \frac{|\xi+\eta|-|\xi|}{|\xi|} \frac{\xi+\eta}{|\xi+\eta|},\quad&\text{if }0<|\xi| \leq \delta, \\
0,\quad&\text{if }|\xi|>\delta, \\
\end{cases}
\end{equation*}
[$c>0$ is a positive constant], which has a singularity in  $\xi=0$.


 
If we take the initial condition $u_{0}(x)=\epsilon x $, $\epsilon>0$,  the theoretical solution is (see~\cite{MadenciOterkus})
\begin{equation*}
u_x(x,t)=\frac{8 \epsilon L}{\pi^{2}} \sum_{k=0}^{\infty} \frac{(-1)^{k}}{(2 k +1)^{2}} sin \biggl( \frac{(2k+1) \pi x}{2L} \biggr) cos \biggl( \sqrt{\frac{E}{\rho}} \frac{(2k+1)\pi}{2L}t \biggr)
\end{equation*}

In Table~\ref{tab:error-nonlinear}, we report the maximum errors by varying the spatial and time discretization steps. We can see how all methods become of the first order of accuracy due to the singularity of the pairwise function force and because of the linearization of the function $f$.

\begin{table}
\centering
\renewcommand\arraystretch{1.3}
\begin{tabular}{|c|c|c|c|c|c|c|}
\hline
\hline
\rule[-4mm]{0mm}{1cm}
 \textbf{Methods} &  $h$  & $\tau$ &$N$ & $N_T$ &  $||     {\bf e}|| _{\infty}$ & $\log_2{(R_n)}$   \\
\hline
\hline
 & $0.1000$  & $0.0100$ & $10$& $1000$& $5.4590\times10^{-2}$& -  \\
 MSV & $0.0500$   & $0.0050$ & $20$& $2000$& $2.7285\times10^{-2}$& $1.0007$\\
 & $0.0250$  & $0.0025$& $40$& $4000$& $1.3605\times10^{-2}$ & $1.0007$\\
\hline
 & $0.1000$ & $0.0100$ & $10$& $1000$&$5.3895\times10^{-2}$  & -  \\
 MMI & $0.0500$  & $0.0050$ & $20$& $2000$&$2.7281\times10^{-2}$ & $0.9819$\\
 & $0.0250$ & $0.0025$ & $40$& $4000$&  $1.3603\times10^{-2}$ & $1.0036$\\
\hline
\hline
\end{tabular}
\renewcommand\arraystretch{1}
\caption{With reference to Test 5: the comparison among the performance of MSV and MMI methods in the nonlinear case by varying $h$, $\tau$, $N$ and $N_T$.}
\label{tab:error-nonlinear}
\end{table}



\section{Conclusions and future work}\label{sec:con}
 In this paper we have considered the linear peridynamic equation of motion which is described by a second order in time partial integro-differential equation. 
We have analyzed numerical techniques of higher order in space to compute a numerical solution, moreover, we have seen how applying similar techniques to the nonlinear model. 
Also a spectral method to discretize the space domain has been discussed. 
Thanks to the numerical simulations, we can deduce that it is possible to treat the linear problem in a not expensive way by implementing the St\"ormer-Verlet method, which is of the second order and is conditionally stable. While, a greater accuracy can be achieved by using Gauss two points formula for space discretization and the Trigonometric scheme for time discretization. Spectral techniques perform very well in the linear case, but they require to deal with periodic boundary conditions. Additionally, all the implemented methods can be applied to the nonlinear case using a linearization of the pairwise force $f$. Also spectral methods can be extended to nonlinear problem, and it could be the aim of future works.
 Furthermore, in future we would apply similar techniques to the nonlinear model using interpolation of the nonlinear terms in order to improve the accuracy in space and extend the results to space
domains of dimension greater than $1$, using finite element methods or mimetic finite difference methods (see for example~\cite{Lopez_Vacca_2016, Beirao_Lopez_Vacca_2017}).
 
\section*{Acknowledgements}
This paper has been partially supported by GNCS of Italian Istituto Nazionale di Alta Matematica. GMC, FM and SFP are members of the Gruppo Nazionale per l'Analisi Matematica, la Probabilit\`a e le loro Applicazioni (GNAMPA) of the Istituto Nazionale di Alta Matematica (INdAM). The authors thank the anonymous referees for the careful reading of the manuscript.

\bigskip
\noindent
{\bf References}


\end{document}